\documentclass[12pt]{amsart}
\usepackage{amsmath, amssymb}
\newfont {\cyr} {wncyr10}
\renewcommand{\labelenumi}{{(\roman{enumi})}}

\usepackage{amsfonts}
\usepackage{amsmath}
\usepackage{amssymb}
\usepackage{setspace}
\usepackage{multicol}
\usepackage{color}

\mathchardef\tnode="020E

\def\arc{ 
 \hbox{\kern -0.15em \vbox{\hrule width 2.5em height 0.6ex depth -0.5 ex} \kern
-0.33em}}

\def\darc{
 \rlap{\lower0.2ex\arc}{\raise0.2ex\arc}}

\def\tarc{
 \rlap{\rlap{\lower0.4ex\arc}{\raise0.4ex\arc}}{\arc}}

\def\stroke#1{
 \kern 0.05
\rlap\arc{{\textstyle{#1}}\atop\phantom\arc} \kern -0.22em}

\def\dstroke#1{
 \kern 0.05em
\rlap\darc{{\textstyle{#1}}\atop\phantom\darc} \kern -0.22em}

\def\centerscript#1{
 \setbox0=\hbox{$\tnode$} \hbox to
\wd0{\hss$\scriptstyle{#1}$\hss}}

\def\node{
 \def\super{} \def\sub{}
\futurelet\next\dolabellednode}

  \let\sp=^ \let\sb=_

  \def\dolabellednode{%
   \ifx\next\sb\let\next\getsub \else \ifx\next\sp\let\next\getsuper
\else\let\next\donode \fi \fi \next}

  \def\getsub_#1{\def\sub{#1}\futurelet\next\dolabellednode}

\def\getsuper^#1{\def\super{#1}\futurelet\next \dolabellednode}

  \def\donode{%
   \rlap{$\mathop{\phantom\tnode}\limits_{\centerscript{\sub}}
^{\centerscript{\super}}$}\tnode}

\def\varcdn{
 \kern
-0.03em\vbox{\kern -0.5ex \hbox to \wd0{\hss\vrule width 0.04em depth 5.8ex\hss} \kern -0.3ex \hbox{$\tnode$}}}

\def\a3{\node_1\arc\node_2\arc\node_3}

\def\c3{\node_1\arc\node_2\darc\node_3}

\def\m24{\node\arc\node\dstroke{\sim}\node} \def\u43{\node\darc\node\dstroke{\sim}\node}

\newcommand{\varcdnl}[1]{ 
\kern -0.03em\vbox{\kern -0.5ex \hbox to \wd0{\hss\vrule width 0.04em depth 5.8ex\hss} \kern -0.3ex
\hbox{$\tnode^{#1}$}}}

\def\d4{\node^1\arc\node^{2}_{\varcdnl{3}}\arc\node^{4}}

\def\nodef{
\def\super{} \def\sub{} \futurelet\next\dolabellednodef}

  \let\sp=^ \let\sb=_

  \def\dolabellednodef{%
  \ifx\next\sb\let\next\getsubf \else
\ifx\next\sp\let\next\getsuperf \else\let\next\donodef \fi \fi \next}

\def\getsubf_#1{\def\sub{#1}\futurelet\next\dolabellednodef}

\def\getsuperf^#1{\def\super{#1}\futurelet\next \dolabellednodef}

  \def\donodef{%
  \rlap{$\mathop{\phantom\tnodef}\limits_{\centerscript{\sub}}
^{\centerscript{\super}}$}\tnodef}

\def\varcdnf{
 \kern -0.03em\vbox{\kern -0.5ex \hbox to \wd0{\hss\vrule width 0.04em depth
5.8ex\hss} \kern -0.3ex \hbox{$\tnodef$}}}

\newtheorem{theorem}{Theorem}[section]
\newtheorem{lemma}[theorem]{Lemma}
\newtheorem{proposition}[theorem]{Proposition}

\newcounter{claim}[theorem]

\def \diag{\mathrm {diag}}

\newcounter{cclaim}[theorem]

\def \udot {{}^{\textstyle .}}

\newcommand{\E}{\mathrm{E}}\newcommand{\SU}{\mathrm{SU}}
\newcommand{\F}{\mathrm{F}}\newcommand{\B}{\mathrm{B}}
\newcommand{\M}{\mathcal{M}}
\newcommand{\G}{\mathrm{G}}

\newcommand{\Q}{\mathrm{Q}}

\newcommand{\Aut}{\mathrm{Aut}}

\newcommand{\Out}{\mathrm{Out}}

\newcommand{\Syl}{\mathrm{Syl}}\newcommand{\syl}{\mathrm{Syl}}

\newcommand{\GF}{\mathrm{GF}}
\newcommand{\GL}{\mathrm{GL}}
\newcommand{\Sp}{\mathrm{Sp}}
\newcommand{\SL}{\mathrm{SL}}
\newcommand{\0}{\emptyset}
\newcommand{\PGL}{\mathrm{PGL}}
\newcommand{\PSL}{\mathrm{PSL}}\newcommand{\PSp}{\mathrm{PSp}}
\newcommand{\Sym}{\mathrm{Sym}}
\newcommand{\Alt}{\mathrm{Alt}}
\newcommand{\Dih}{\mathrm{Dih}}
\newcommand{\SDih}{\mathrm{SDih}}

\newcommand{\U}{\mathrm{PSU}}

\def \qedc {$\hfill \blacksquare$\newline}

\def \L {\hbox {\rm PSL}}

\def \OO {\hbox {\rm O}}

\def \GU {\mbox {\rm GU}}

\def \syl {\hbox {\rm Syl}}\def \Syl {\hbox {\rm Syl}}

\def \ov {\overline}

\def \Aut{ \mathrm {Aut}}

\def \Out{\mbox {\rm Out}}

\def \Mat{\mbox {\rm Mat}}

\def \B{\mbox {\rm B}}

\def \M{\mbox {\rm M}}

\def \Co {\mbox {\rm Co}}

\def \McL{\mbox {\rm McL}}

\def \PSU {\mbox {\rm \PSU}}
\def \GSp {\mbox {\rm GSp}}

\begin{document}
\renewcommand{\labelenumi}{(\roman{enumi})}

\title  {Groups which are almost groups of Lie type in characteristic $p$}
 \author{Chris Parker}
  \author{Gernot Stroth}
\address{Chris Parker\\
School of Mathematics\\
University of Birmingham\\
Edgbaston\\
Birmingham B15 2TT\\
United Kingdom} \email{c.w.parker@bham.ac.uk}

\address{Gernot Stroth\\
Institut f\"ur Mathematik\\ Universit\"at Halle - Wittenberg\\
Theordor Lieser Str. 5\\ 06099 Halle\\ Germany}
\email{gernot.stroth@mathematik.uni-halle.de}

\email {}

\date{\today}

\maketitle \pagestyle{myheadings}

\markright{{\sc }} \markleft{{\sc Chris Parker and Gernot Stroth}}
\begin{abstract} For a prime $p$, a $p$-subgroup of a finite group $G$ is said to be large if and only if $Q= F^*(N_G(Q))$
and, for all $1 \neq U \le Z(Q)$, $N_G(U) \le N_G(Q)$. In this article we determine those groups $G$ which have a large subgroup and which in addition have a proper subgroup $H$ containing a Sylow $p$-subgroup of $G$ with $F^*(H)$  a group of Lie type in characteristic $p$ and rank at least $2$ (excluding $\PSL_3(p^a)$) and $C_H(z)$ soluble for some $z \in Z(S)$. This work is part of a project to determine the groups $G$ which contain a large $p$-subgroup. \end{abstract}
\section{Introduction}

When classifying groups of Lie type in characteristic $p$, $p$ a prime,  one usually tries to determine the parabolic subgroups, construct a chamber system from the parabolic subgroups  and finally identify the groups via the classification
of the corresponding buildings. The parabolic subgroups in a group of Lie type are examples of $p$-local subgroups where, in an arbitrary group $G$, a $p$-local subgroup is by definition the normalizer of a non-trivial $p$-subgroup of $G$.
Hence the first step in a classification theorem whose target groups are predominantly groups of Lie type in characteristic $p$ should be to  determine the structure  of the maximal $p$-local subgroups containing a fixed Sylow $p$-subgroup. This approach has been started in a  paper  by Meierfrankenfeld, Stellmacher
and the second author  \cite{MSS} where groups with a large $p$-subgroup are studied.
Here, given a group $G$,  a $p$-subgroup $Q$ of $G$  is \emph{large} if and only if
\medskip
\begin{enumerate}
\item[(L1)]\label{1} $Q = F^*(N_G(Q))$; and
\item[(L2)]\label{2} if $U$ is a non-trivial subgroup of $ Z(Q)$, then $N_G(U)\le N_G(Q)$.
\end{enumerate}
\medskip
We will frequently use the fact that  condition (L1) is equivalent to $Q=O_p(N_G(Q))$ and $C_G(Q)\le Q$.

A motivating  observation is that most of the groups of the Lie type in characteristic $p$ contain a  \emph{large} $p$-subgroup.
For example, in $\PSL_n(p^a)$, with $n \ge 3$, conjugates of the radical subgroups of the point and
hyperplane stabilizers are large. If we consider $\PSp_{2n}(p^a)$ with $p$ odd and $n \ge 2$, then the large
subgroups are the radical subgroup of the stabilizer of a point. The group $\PSp_{2n}(2^a)$, $n \ge 2$, has no
large $2$-subgroups as $\PSp_{2n}(2^a)$ is generated by the centralizers of a long and a short root element contained in the centre of a fixed Sylow $2$-subgroup.

 We say that a $p$-local subgroup $M$ of a group $G$ is of \emph{characteristic $p$} provided $F^*(M) = O_p(M)$ and we say that $G$ is of
\emph{parabolic characteristic $p$} provided all $p$-local subgroups of $G$ which contain a Sylow $p$-subgroup are of characteristic $p$.

A particularly appealing  consequence of a group $G$ containing a large $p$-subgroup is that  $G$ then is a group of parabolic
characteristic $p$ (Lemma~\ref{basics} (iv)). This means that the existence of a large $p$-subgroup in a simple group is  an indication
that the group may be a Lie type group defined in characteristic $p$.

The contributions  in \cite{MSS}  begin the identification of those groups which contain a large $p$-subgroup $Q$ for
some prime $p$. If we consider groups $G$ of Lie type and Lie rank at least two which contain a large subgroup $Q$ and fix a Sylow $p$-subgroup $S$
of $G$  with $Q \leq S$, then there is a maximal $p$-local subgroup $M$ containing $S$ such that $Q \not\leq
O_p(M)$. So to mimic this higher Lie rank assumption in \cite{MSS} it is assumed that there is such an $M$ in general. Then the aim of
the work in \cite{MSS} is to provide information about the structure of the maximal $p$-local subgroups of $G$
which do not normalize $Q$.

Once this goal is achieved one possible plan is to proceed as follows. Let $M$ be  a maximal $p$-local subgroup
of $G$ with $M \not \le N_G(Q)$ and $C \leq N_G(Q)$ be minimal such that $C \not\leq N_M(Q)$. Then set $H =
\langle M, C \rangle$. Now in the typical case one is able to show that $H$ is an automorphism  group of a group
of Lie type in characteristic $p$, using the approach described in the first paragraph of this introduction. In
fact in most cases $H$ is the target group. Hence the remaining difficulty is to prove that $H = G$.

This difficulty can often be  overcome as follows. First  show that  $N_G(Q) \leq H$ and then  using this
demonstrate that $N_G(E) \leq H$ for all non-trivial $p$-subgroups $E$ which are normal in some Sylow
$p$-subgroup of $H$. Having achieved this, show that $H$ is a strongly $p$-embedded subgroup of $G$ and
use this to reach  the conclusion that $H=G$ with the help of \cite{PSStrong}.

When  $p$ is odd,  A. Seidel in his PhD thesis \cite{Seidel}  has shown that the first two steps can be taken whenever the Lie rank of $H$ is at least 3 and $N_H(Q)$ is not soluble.  In work in progress G. Pientka is tasked to prove the same result  when $p = 2$.
Hence the open question is: \emph{what are the obstructions to taking the first two steps when  $N_H(Q)$ is soluble?} The purpose of this paper is to answer this question. It is remarkable how often this situation arises. Precisely we prove

\begin{theorem}\label{main} Assume that $p$ is a prime,  $G$ is a group,  $H=N_G(F^*(H))$  contains a Sylow $p$-subgroup
of $G$ and $F^\ast(H)$ is a simple group of Lie type in characteristic $p$ and rank at least two but that $F^*(H)
\not \cong \PSL_3(p^a)$ when  $p$ is odd.

Suppose that  a large subgroup $Q$ of $G$ is contained in $H$  and  $C_H(z)$ is soluble for some $p$-central element $z$ of $G$.
Then one of the following holds:
\begin{enumerate}
\item  $N_G(Q)= N_H(Q)$;
\item    $p=2$ and $F^\ast(G) \cong \Mat(11),  \Mat(23)$, $\G_2(3)$ or $\mathrm P\Omega^+_8(3)$; or \item  $p=3$ and
$F^\ast(G) \cong  \U_6(2)$, $\F_4(2)$, ${}^2\E_6(2)$,  $\McL$,  $\Co_2$, $\M(22)$,  $\M(23)$ or $\F_2$.
\end{enumerate}
\end{theorem}

We remark that the proof of Theorem~\ref{main} does not require a hypothesis, such as the $\mathcal K$-group
hypothesis, on the composition factors of proper subgroups.

Regarding the omitted cases when $G = \PSL_3(p^a)$ with $p$ odd and $p^a >
13$ we expect that it can be shown that $N_G(Q)= N_H(Q)$ (see \cite[Theorem 1.5]{BNpairs} to see why this should be the case).
However, in the case $p^a\le 13$,    there are serious problems as the configuration we are examining is close to a configuration in the O'Nan simple group when $p=7$ and the Monster simple group when $p=13$.

In Section 2, we present various preliminary results that will be used in the proof of the main theorems. In
particular, we produce the (well-known) list of simple Lie type groups defined in characteristic $p$ of rank at least two in which the
centralizer of some $p$-central element (a non-trivial element contained in the centre of a Sylow $p$-subgroup)  is soluble. It
transpires that this can only occur when either the rank of $H$ is two or when $p^a \in \{2,3\}$. Section 2  also
contains amalgam type characterizations of the  simple groups $\Mat(22)$ and $\Mat(23)$.

The proof of Theorem~\ref{main} is presented in Sections 3 and 4 where we deal with the configurations which
arise when $p=2$ and $p=3$ respectively. When $p=2$, the most troublesome situation arises when $H= \PSL_3(2^n)$
with $n \ge 2$. This is the situation which ultimately leads to the group $\Mat(23)$ and is close to
configurations which exist in other simple groups which, however, fail to have a large subgroup. A common feature
in the analysis is that  $Q$ often turns out to be an extraspecial $p$-group. In this case the possibilities for
the structure of $N_G(Q)/Q$ can be determined as the outer automorphism group of such a group is either an
orthogonal group of the appropriate type if $p=2$ and if $p$ is odd and $Q$ has exponent $p$ then it is a general
symplectic group \cite[20.5]{DoerkHawkes}.
 The overall strategy of the proof is to determine the possible structure of $N_G(Q)$ (where $Q$ is the large subgroup)
and then once this is done use characterization theorems to identify the groups from either $2$-local or
$3$-local information.  As an illustrative example, consider the possibility that $H \cong \PSL_4(3)$ or
$\U_4(3)$. In this case we show that $Q$ is an extraspecial group of order $3^5$ and then, using the subgroup
structure of $\Out(Q) \cong \GSp_4(3)$, we show that $N_G(Q)/Q$ has restricted structure. We then further
investigate the $3$-local structure of $G$ until we have a good approximation to the structure of $N_G(Q)/Q$ at
which stage we cite the appropriate recognition theorems \cite{McL,Co2, F42, U62}.

Throughout this article we follow  Atlas \cite{Atlas} notation for group extensions. Indeed the Atlas is a good source for readers unfamiliar with the subgroup structure of the small simple groups to extract various pieces of useful information about the groups we shall encounter. Our group theoretic notation
is mostly standard and follows that in  \cite{GLS2} for example. For a prime $p$, we say that a non-trivial element is $p$-central provided its centralizer contains a Sylow $p$-subgroup of $G$.

For odd primes $p$, the extraspecial groups of exponent $p$ and order $p^{2n+1}$ are denoted by $p^{1+2n}_+$. The
extraspecial $2$-groups of order $2^{2n+1}$ are denoted by $2^{1+2n}_+$ if the maximal elementary abelian
subgroups have order $2^{1+n}$ and otherwise we write $2^{1+2n}_-$. We expect our notation for specific groups is
self-explanatory. For a subset $X$ of a group $G$, $X^G$ denotes that set of $G$-conjugates of $X$.   Often we
shall give suggestive descriptions of groups which indicate the isomorphism type of certain composition factors.
We refer to such descriptions as the \emph{shape} of a group. Groups of the same shape have normal series with
isomorphic sections.   We use the symbol $\approx$ to indicate that two groups have the same shape.

\bigskip

\noindent {\bf Acknowledgement.}  The first author is  grateful to the DFG for their support and thanks the
mathematics department in Halle for their hospitality.

\section{Preliminary results}
%
%

In this section we present various results that we need for the proof of Theorem~\ref{main}.
The basic lemma about large subgroups that we shall use (without further reference) is as follows.

\begin{lemma}\label{basics} Suppose that $Q$ is a large $p$-subgroup of $G$ and $T$ is a non-trivial $p$-subgroup of $G$ such
that $N_G(T) \ge Q$. Then
\begin{enumerate}
\item  $N_G(Q)$ contains the normalizer of a  Sylow $p$-subgroup of $G$.
\item If $Q\le T$, then $N_G(T) \le N_G(Q)$.
\item  $F^*(N_G(T))= O_p(N_G(T))$.
\item  $G$ has parabolic characteristic $p$.
 \end{enumerate}
\end{lemma}

\begin{proof} Suppose that $S \in \syl_p(G)$ and $Q \le S$. Then, by property (L1), $Z(S) \le Z(Q)$ and, by property (L2),  $N_G(S)$ normalizes $Q$. So (i) holds.

Suppose that $x \in N_G(T)$ and $Q \le T$. Then $Q^x \le T \le S$. Therefore $N_G(Q)$ and $N_G(Q^x)$ both contain $N_G(S)$ by (i). It follows from Sylow's Theorem that $N_G(Q)= N_G(Q^x)$ and then $Q=Q^x$ by (L1).

Suppose that $Q \le N_G(T)$ and set $A = E(N_G(T))O_{p'}(N_G(T))$. Then $Q$ normalizes $A$. Since $A$ centralizes
$C_T(Q) \le Z(Q)$, (L2) implies that $A$ normalizes $Q$. But then $[Q,A]\le Q \cap A$ is a normal $p$-subgroup of
$A$, whence $[Q,A] \le Z(A)$ and $[Q,A]=[Q,A,A]=1$.  Thus $A \le Q$ by (L1) and this means that $A=1$. Therefore (iii)
holds and (iv) is a direct consequence of (iii).
\end{proof}

\begin{lemma}\label{AB} Suppose that $p$ is a prime, $G$ is a group, $S \in \syl_p(G)$ has order $p^a$ and $V$ is a
faithful $\GF(p)G$-module of dimension $2a$. If, for all $s \in S^\#$, $$C_V(S)=[V,S]=C_V(s)=[V,s]$$ has
dimension $a$, then either $|\Syl_p(G)|=1$ or $|\syl_p(G)|=p^a+1$.
\end{lemma}

\begin{proof} See \cite[Lemma 4.16]{BNpairs}.
\end{proof}

The following lemma, which we will use several times,  is an easy consequence of the Three Subgroup Lemma.
\begin{lemma}\label{w=1} Suppose that $p$ is a prime, $P$ is a $p$-group of nilpotency class at most $2$ and that $\alpha\in \Aut(P)$ has order coprime to $p$.
If $\alpha$ centralizes a maximal abelian subgroup of $P$, then $\alpha=1$.
\end{lemma}

\begin{proof}  As $\alpha$ has order coprime to $p$, we have $P=[P,\alpha]C_P(\alpha)$. Suppose that $E \le C_P(\alpha)$ is a maximal abelian subgroup of $P$. Then $Z(P) \le E$ and so is
centralized by $\alpha$. Thus, as $P$ has nilpotency class at most $2$, $$[[P,C_P(\alpha)],\alpha]\le [Z(P),\alpha]=1.$$
Because we also have $[[C_P(\alpha),\alpha],P]=1$, the Three Subgroup Lemma yields $[[P,\alpha],C_P(\alpha)]=1$.
Since $C_P(E)=E$, we then have $[P,\alpha]\le E\le C_P(\alpha)$. Thus $P=[P,\alpha]C_P(\alpha)= C_P(\alpha)$ and
this proves the result.
\end{proof}

{\begin{lemma}\label{possible} Suppose that $H$ is a simple group of Lie type defined in
characteristic $p$ of rank at least $2$. Assume that $C_H(z)$ is soluble for some $p$-central element  of $H$. Then  one of the
following holds:
\begin{enumerate}
\item $H \cong \PSL_3(p^a)$ for some $a \ge 1$;
\item  $p=2$ and $H\cong   \Sp_6(2)$, $\U_4(2) \cong \PSp_4(3)$, $\U_5(2)$,
$\G_2(2)'\cong \U_3(3)$, ${}^2\F_4(2)$, $\Omega_6^+(2)\cong \SL_4(2)$, $\Omega_8^+(2)$ or $\Sp_4(2^a)'$ for some $a\ge 1$; or
\item $p=3$ and $H \cong \PSp_4(3)\cong \U_4(2)$, $\PSL_4(3)$, $\U_4(3)$, $\Omega_7(3)$,
$\mathrm P\Omega_8^+(3)$ or $\G_2(3^a)$ for some $a \ge 1$.\end{enumerate}
\end{lemma}}

\begin{proof}  Let $S \in \syl_p(H)$ and  $n$ represent the rank of $H$. Then
either $Z(S)$ is a long root group or $H \cong \Sp_{2n}(2^a)'$, $\F_4(2^a)$ or $\G_2(3^a)$ for $a \ge 1$ and $Z(S)$ is the
product of the root groups corresponding to the highest long root and the highest short root (\cite[Theorem
3.3.1]{GLS3}). Using \cite[Theorem 3.2.2]{GLS3}  it is easy to see that if $z\in Z(S)$ is a long root element and
if  $p^a> 3$  and  $n\ge 3$, then $C_H(z)$ is non-soluble.

Suppose that $n=2$ and that $H \not \cong \PSL_3(p^a)$. Let $z \in Z(S)$ be a long root element.  If $H \cong
\PSp_4(p^a)$ with $p^a>3$ and odd, then $C_H(z)$  is non-soluble. The groups $\PSp_4(3)$ and $\PSp_4(2^a)'$ are
 listed in (ii) and (iii). If $G \cong \U_4(p^a)$ or $\U_5(p^a)$, then  $C_H(z)$  contains a section isomorphic
to $\PSL_2(p^a)$ or $\U_3(p^a)$ respectively. Hence $C_H(z)$  is non-soluble if $a\ge 2$ or $p^a=3$ and $H \cong
\U_5(3)$.   Thus $\U_4(2)$ and $\U_5(2)$ are included in (ii) and $\U_4(3)$ is listed in (iii). If $G \cong
\G_2(p^a)'$, then $C_H(z)$ contains a section isomorphic to $\PSL_2(p^a)$ and so is non-soluble unless $p^a \in
\{2,3\}$. The non-root elements in $Z(S)$ when $H\cong \G_2(3^a)$ have centralizer contained in the normalizer of
$N_H(S)$. So (ii) lists $\G_2(2)'$ and (iii) includes $\G_2(3^a)$ for all positive $a$. If $H \cong
{}^2\F_4(2^{2a+1})$, then $C_H(z)$ contains a section isomorphic to ${}^2\B_2(2^{2a+1})$ and is thus non-soluble
if $a>1$ and ${}^2\F_4(2)'$ is itemized in (ii). This completes the analysis when $n=2$.

So assume that $n\ge 3$ and $p\in \{2,3\}$.  If $Z(S)$ is not a root group then $p=2$ so we consider this case
first.   If $H \cong \F_4(2^a)$, then $C_H(Z(S))$  contains a section isomorphic to $\Sp_4(2^a)'$ and so this
group is not listed. Suppose that $H \cong \Sp_{2n}(2^a)$. Then $C_H(Z(S))$  contains a section isomorphic to
$\Sp_{2(n-2)}(2^a)'$. This group is not soluble if $2^a > 2$ or $n >3$. Hence $\Sp_6(2)$ is listed in (ii). We
may now additionally assume that $Z(S)$ is a root group and $n \ge 3$.

If $n \ge 4$, then $C_H(z)$ is non-soluble  (containing a section of Lie rank at least 2) or $H\cong \mathrm
P\Omega_8^+(p)$ and these groups are included in (ii) and (iii).    We now may assume that the rank of $H$ is $3$
and that $p^a\in \{2,3\}.$

If $p=3$ then we include  $H\cong \PSL_4(3)$,  $\Omega_7(3)$ in (iii) and if $p=2$, we have placed $\SL_4(2)$ in (ii).  When $H\cong \Omega_8^-(p)$,  $C_H(z)$ contains a section isomorphic to $\PSL_2(p^2)$.  The possibilities  $H \cong \U_6(p)$ or $\U_7(p)$ have $C_H(z)$  non-soluble as it has a  section isomorphic to $\U_4(p)$ or $\U_5(p)$ respectively. If  $H \cong \PSp_6(3)$, then $C_H(z)$ contains a section isomorphic to $\PSp_4(3)$. So these latter groups are not included in the conclusion of the lemma.
\end{proof}

We shall need the following specific fact about the normalizer of an extraspecial $2$-subgroup of $\Sp_8(3)$.

\begin{lemma}\label{extra} Suppose that $G \cong \Sp_8(3)$ and $X$ is an extraspecial subgroup of $G$ of order
$2^7$. Then $N_G(X)/X \cong \Omega_6^-(2) \cong \U_4(2)$ and, in particular, $N_G(X)$ contains no elements which act as transvections on $X/Z(X)$.
\end{lemma}

\begin{proof} This follows from \cite[Proposition 4.6.9]{KL}.
\end{proof}

\begin{lemma}\label{VO6} Suppose that $X \cong \Omega_6^-(2) \cong \U_4(2)$, $V$ is
the natural $6$-dimensional $\GF(2)\Omega_6^-(2)$ and let $\mathrm q$ be the associated quadratic form. Suppose
that $Y^* = \mathrm O_2^-(2)\wr \Sym(3)$ and $Y= Y^* \cap X \approx 3^3.\Sym(4)$.  Assume that $F_1$ and $F_2$
are fours groups in $Y$ which are not $Y$-conjugate and that $F_1O_3(Y)$ is normal in $Y$.
\begin{enumerate}
\item For involutions $t \in X$, $\dim C_V(t) = 4$ and $[V,t]$ contains a singular vector.
\item  $\dim C_V(F_1)= \dim C_V(F_2) =3$.
\item $C_V(F_1)$ contains singular vectors.
\item There exists $f \in F_2$ such that $[C_V(f),F_2]$ contains singular vectors.
\end{enumerate}
\end{lemma}
\begin{proof}
The   first assertion is well-known see \cite[Lemma 2.2]{F42}. Let $$\{x_1,x_2,y_1,y_2,z_1,
z_2\}$$ be basis for $V$ such that $\langle x_1,x_2 \rangle$, $\langle y_1,y_2 \rangle$ and $\langle z_1,z_2
\rangle$ are of $-$-type and pairwise perpendicular.  Moreover assume that the decomposition is preserved by $Y$. Then with respect to this basis we may
suppose that $F_1$ is generated by the matrices
$$a=\diag\left(\left(\begin{array}{cc} 1&1\\ 0 &1
\end{array}\right), \left(\begin{array}{cc} 1&1\\0&1
\end{array}\right), \left(\begin{array}{cc} 1&0\\0&1
\end{array}\right)\right)$$ and $$b=\diag\left(\left(\begin{array}{cc} 1&1\\0&1
\end{array}\right), \left(\begin{array}{cc} 1&0\\0&1
\end{array}\right), \left(\begin{array}{cc} 1&1\\0&1
\end{array}\right)\right)$$

Hence $[V,a]=\langle  x_2,y_2\rangle$ and $[V,b]= \langle y_2,z_2\rangle$ which means that $[V,F_1]=C_V(F_1) =
\langle x_2,y_2,z_2\rangle$ and this is clearly an isotropic space. Furthermore, we have $\mathrm q (x_2+y_2) =
q(x_2) + q(y_2)= 1+1=0$. Thus $x_2+y_2$ is singular.

We can assume that $F_2$ is generated by $a$ and $$c = \diag\left(\left(\begin{array}{cccc}
0&0&1&0\\0&0&0&1\\1&0&0&0\\0&1&0&0
\end{array}\right), \left(\begin{array}{cc} 1&0\\0&1
\end{array}\right)\right)$$ (as the elements
in $\mathrm O_6^-(2) \setminus \Omega_{6}^-(2)$ have odd dimensional commutator spaces on $V$) and so $[V,F_2] =
\langle x_2, y_2, x_1+y_1\rangle$ has dimension $3$ and $[C_V(a),F_2] = \langle x_2+y_2\rangle$ is singular. This
completes the proof of the lemma.
\end{proof}

\begin{lemma}\label{O8mod} Suppose that $V$ is an orthogonal module over $\GF(2)$ of dimension $8$ and plus type. Let $X \cong \Sp_2(2) \wr \Sym(3)$ acts on $V$ such that the base group $B \cong \Sp_2(2)\times \Sp_2(2)\times \Sp_2(2)$ acts irreducibly. Assume that the direct factors of $B$ are $B_1$, $B_2$ and $B_3$  with, for $1 \le i \le 3$, $B_i = \langle a_i, b_i\rangle$ where $a_i$ is an involution and $b_i$ has order $3$. Then
\begin{enumerate}
\item $[V,a_1]$ is totally singular of dimension $4$.
\item $[V,a_1a_2]$ is totally singular of dimension $4$.
\item $[V,a_1a_2a_3]$  is totally isotropic  of dimension $4$ but is not totally singular.
\item If $f \in X \setminus B$ has order $2$ and commutes with $B_3$, then $[V,f]$ has dimension $2$.
\item $b_1$, $b_2$, $b_3$ and their inverse are the only elements of order $3$ in $B$ which act fixed point freely on $V$.
\end{enumerate} \end{lemma}

\begin{proof}  There is a unique $8$-dimensional representation for $B$ over $\GF(2)$ and this is obtained as a tensor product of the natural symplectic  modules for $B_1$, $B_2$ and $B_3$. For $i=1,2,3$, let $V_i$ be a natural module for $B_i$ with symplectic basis $\{e_i, f_i\}$  and associated symplectic form $(\;,\;)_i$.
Set $V= V_1 \otimes V_2 \otimes V_3$. Then  $V$ supports a symmetric bilinear form $$(\;,\;) = \prod_{i=1}^3 (\;,\;)_i$$  and an associated quadratic form $\mathrm q$ which is entirely defined by specifying that all the pure tensors are singular.  In this way $B$ embeds into $\mathrm O_8^+(2)$. This form is preserved by $X$ which also has a unique $8$-dimensional irreducible representation (the faithful representation degrees are $6$, $8$, $12$ and $16$).

We may suppose that $a_i\in B_i$ centralizes $e_i$ and sends $f_i$ to $e_i+f_i$. Then $$[V,a_1]= \langle e_1\otimes x\otimes y \mid x \in \{e_2,f_2\}, y \in \{e_3,f_3\}\rangle$$
which is totally singular.
Similarly  $$[V,a_1a_2]= \langle e_1\otimes e_2\otimes y,
e_1\otimes f_2\otimes y+ f_1\otimes e_2\otimes y+f_1\otimes f_2\otimes y\mid  y \in \{e_3,f_3\}\rangle$$
is also totally singular.
Finally

\begin{eqnarray*}[V,a_1a_2a_3]&=& \langle e_1\otimes e_2\otimes e_3,
f_1\otimes e_2\otimes e_3+
e_1\otimes f_2\otimes e_3,
\\&&
e_1\otimes e_2\otimes f_3+
e_1\otimes f_2\otimes e_3,\\&&
e_1\otimes e_2\otimes f_3+f_1\otimes f_2\otimes e_3+
e_1\otimes f_2\otimes f_3+
f_1\otimes e_2\otimes f_3\rangle.
\end{eqnarray*}
We see that this space is isotropic but that
$$\mathrm q(e_1\otimes e_2\otimes f_3+f_1\otimes f_2\otimes e_3+
e_1\otimes f_2\otimes f_3+
f_1\otimes e_2\otimes f_3) = (e_1\otimes e_2\otimes f_3,f_1\otimes f_2\otimes e_3)=1$$ which means that this space is not totally singular.
This  proves (i), (ii) and (iii).
If $f$ is as in part (iv), then $$C_V(f) = \langle e_1 \otimes e_2\otimes y, f_1 \otimes f_2\otimes y, e_1\otimes f_2\otimes y + f_1\otimes e_2\otimes y \mid y \in\{e_3,f_3\}\rangle.$$ So (iv) holds.

For $i=1,2,3$, let $b_i$ be the elements of $B_i$ which maps $e_i$ to $e_i+f_i$ and $f_i$ to $e_i$. Then we see that $b_1b_2$ centralizes $e_1\otimes f_2\otimes f_3 + f_1 \otimes e_2 \otimes f_3$ and that $b_1b_2b_3$ centralizes $$e_1\otimes e_2\otimes e_3+(e_1+f_1)\otimes (e_2+f_2)\otimes (e_3+f_3) +f_1\otimes f_2\otimes f_3.$$ It is also simple to see that $[V,b_1] = V$. Thus (v) holds.
\end{proof}

To deal with the possibility that $H^* \cong {}^2\F_4(2)'$ we will need the following facts about this group.

\begin{lemma}\label{2f4p} Suppose that $X \cong {}^2\F_4(2)$ and that $X^*$ is the derived subgroup of $X$.
Let $T \in \Syl_2(X)$, $Z= Z(T)$, $V= Z_2(T)$, $P= C_G(Z)$, $R = O_2(P)$ and $M= N_G(V)$. For a subgroup $Y$ of
$X$ set $Y^*= Y \cap X^*$.
\begin{enumerate}
\item The $2$-rank of $X$ and the $2$-rank of $X^*$ are both  $5$.
\item $M/O_2(M) \cong \SL_2(2)$, $V$ has order $4$ and is a natural $M/O_2(M)$-module.
\item $W = Z_3(T)$ has order $8$ and  is elementary abelian. $M$ normalizes $Z_3(T)$ and  $M/C_M(Z_3(T))\cong \Sym(4)$. Furthermore $|R : C_R(W)| = 4$.
\item $Z_4(T)$ has order $16$.
\item $P/R \cong P^*/R^*$ is isomorphic to a Frobenius group of order 20.
\item $Z= Z(R^*)$ has order $2$, $Z_2(R) = Z_2(R^*)$ is elementary abelian of order $2^5$ and $Z_2(R)/Z$ is
a $P$-chief factor.
\item $R/C_R(Z_2(R))$ is a $P$-chief factor.
\item $C_{R}(Z_2(R))$ is an abelian group of order $2^6$ and $\Omega_1(C_{R}(Z_2(R)))= Z_2(R)$.

\end{enumerate}
\end{lemma}

\begin{proof} That the $2$-rank of $X$ and $X^*$ is $5$ is given in \cite[Theorem 3.3.3]{GLS3}.

We use the results and notation from \cite{GLS3} especially Corollary 2.4.6 and the passages on pages 101 and
102. Thus we have root groups $X_1$ to $X_{16}$ with $X_i$ of order $2$ if $i$ is even and cyclic of order $4$ if
$i$ is odd. For odd $i$ we define $Y_i = \Omega_1(X_i)$. The opposite root group of $X_i$ is $X_{i+8}$ for $1\le i\le 8$.
We have $T= \prod_{i=1}^8 X_i$, $P= \langle T, X_9\rangle$, $M= \langle T, X_1\rangle$, $R= \prod_{i=2}^8{ X_i}$
and $O_2(M)=\prod_{i=1}^7{ X_i}$. In addition $M_1=\langle X_8, X_{16}\rangle \cong \SL_2(2)$ and $P_1= \langle
X_1,X_9\rangle \cong {}^2\B_2(2)$ is the Frobenius group of order $20$.

We calculate that $Z(T)= Y_5$, $V= Z_2(T) = Y_5Y_3$, $Z_3(T) = Y_5Y_3X_4$, $Z_4(T)=Y_5Y_3X_4X_6$ using
\cite[Theorem 2.4.5 (d) and Corollary 2.4.6]{GLS3}. Further, using the same results,  we verify that $V$ and $W$
are normalized by $M$ and so this proves (ii), (iii) and (iv).

Using the statement in \cite[pages 101 and 102]{GLS3}, we get  $Z(R)= Y_5$, $Z_2(R) = Y_5Y_3X_4X_6Y_7$ is elementary
abelian of order $2^5$. We have $C_{R}(Z_2(R))= Z_2(R)X_5$ and so this is an abelian group of order $2^6$ which
is not elementary abelian as $X_5$ is cyclic of order $4$. Also from \cite[pages 101,102]{GLS3} we have that
$R/C_R(Z_2(R))$  and $Z_2(R)/Z(R)$ are $P$-chief factors. In particular, as  $[R,R] = Z_2(R) \le X^*$,  $Z_2(R)=
Z_2(R^*)$. Finally $X^*$ contains the elements $x_1(1)x_3(1)$ and $y_9(1)$ by \cite[Theorem 3.3.2]{GLS3} and
modulo $R$ these two elements generate the ${}^2\B_2(2)$. Thus $P^*/R^*\cong {}^2\B_2(2)$. This discussion
demonstrates (v), (vi), (vii) and (viii).
\end{proof}

We shall also need the two following elementary lemmas.

\begin{lemma}\label{SU43} Suppose that $X \cong \SU_4(3)$, $T\in \Syl_3(X)$ and $M= N_X(Z(T))$. Then $M$ acts
irreducibly on $O_3(M)/Z(T)$.
\end{lemma}

\begin{proof} We know that $O_3(M)$ is extraspecial of order $3^5$ and  $M/O_3(M)$ acts faithfully on $O_3(M)/Z(T)$ and contains a subgroup isomorphic to
$\SL_2(3)$ at index $2$. Now we note that the diagonal subgroup of $X$ is homocyclic of order $16$, it follows
that $M/O_3(M)$ is not isomorphic to $\GL_2(3)$ and hence the action of $M$ on $O_3(M)$ must be irreducible.
\end{proof}

The next lemma exhibits some exceptional behaviour of $\SL_4(3)$.

\begin{lemma}\label{SL43}  Suppose that $X \cong \SL_4(3)$, $T\in \Syl_3(X)$ and $M= N_X(Z(T))$. Then $M$ contains exactly
four normal subgroups of order $27$. Two of them are elementary abelian and two of them are extraspecial.
\end{lemma}

\begin{proof} We note that as $C_X(Z(T))/O_3(M) \cong \SL_2(3)$, there are at most $4$ normal subgroups of order $27$ in $M$.
The two elementary abelian subgroups are easy to see.    Take $Z(T) $ to be generated by the matrix with ones
down the diagonal and ones in the bottom left corner, the two extraspecial normal subgroups of $M$ are  generated
by
$$\left(\begin{array}{cccc} 1&0&0&0\cr 1&1&0&0\cr 0&0&1&0\cr 0&0&2&1 \end{array}\right),\left( \begin{array} {cccc}1&0&0&0\cr
0&1&0&0\cr 1&0&1&0\cr 0&1&0&1
\end{array}\right)$$ and
$$\left(\begin{array}{cccc} 1&0&0&0\cr 1&1&0&0\cr 0&0&1&0\cr 0&0&1&1 \end{array}\right),\left( \begin{array}
{cccc}1&0&0&0\cr 0&1&0&0\cr 1&0&1&0\cr 0&2&0&1
\end{array}\right).$$

\end{proof}

\def \a {\alpha}
\def \b {\beta}

\begin{lemma}\label{M22} Suppose that $G$ is a group and $P$, $B$ and $L$ are subgroups of $G$ such that $P \cong \PSL_3(4)$, $P\cap B \cap L$ is a Borel subgroup of $P$ and $P \cap B$ and $P\cap L$ are point and line stabilizers in $P$ respectively. Assume that $B \approx 2^4:\Alt(6)$, $L \approx  2^4:\Sym(5)$ and $|L:L\cap B|=5$. Then $\langle P,B,L\rangle \cong \Mat(22)$.
\end{lemma}

\begin{proof}
We may as well suppose that $G= \langle P,B,L\rangle$. We consider the graph $\Gamma$ which has vertex set
$\mathcal P \cup \mathcal B$ where  $\mathcal P=\{Pg\mid g \in G\}$ and $\mathcal B = \{Bg\mid g\in G\}$ and edge
set consisting of just the pairs $\{Pg,Bh\}$ such that $Pg\cap Bh \not =\emptyset$. The group $G$ acts on
$\Gamma$ by right multiplication and the kernel of the action is a normal subgroup of $P$ contained in $P\cap B$.
As $P$ is a simple group this means that the action of $G$ is faithful.   Since $L =  (P\cap L)(L\cap B)$ the
stabilizer of the connected component containing $P$ and $B$ is $G=\langle P,B\rangle$ and therefore $\Gamma$ is
connected. If $L$ normalizes $P$, then, as $B=\langle B\cap P, B \cap L\rangle$,  $B$ also normalizes
$P$. However $B \cap P$ is not normal in $B$ and so this is impossible. Furthermore, as $P$ acts on $\mathcal P$,
we have that $|\mathcal P| \ge 22$.

 For $\alpha \in \Gamma$ we let $\Gamma(\alpha)$ denote the set of neighbours of $\a$ in $\Gamma$. The pointwise stabilizer in $G$ of a subset $\Theta$  of $\Gamma$  is written as $G_\Theta$. Note that if $\a= P$, then $G_{\a}= P$
and if $\b = B$ then $G_\b = B$ and the other stabilizers are conjugates of these groups. Our first observation
is obvious.  Let $\a = P $ and $\b = B$. Then $|\Gamma(\a)|=|P:P\cap B|=21$ and $|\Gamma(\b)|=|B:P\cap B|= 6$.
Moreover $G_\a$ acts on $\Gamma(\a)$ as it acts on the points of the projective plane and $G_\beta$ acts as
$\Alt(6)$ on $\Gamma(\beta)$.

Now $G_{\a\b}=P\cap B \approx 2^4:\Alt(5)$ acts transitively $\Gamma(\beta)\setminus\{\a\}$ and $G_{\a\b\gamma}
\approx 2^4:\Alt(4)$ for any $\gamma \in \Gamma(\beta)\setminus\{\alpha\}$. Let $x \in (L\cap B) \setminus P$ and set $\gamma= Px$. Then $Px \cap B = Px \cap Bx =
(P\cap B)x$ is non-empty. Thus $\gamma \in \Gamma(\beta)$. Furthermore, $P\cap P^x = P \cap L$ as $P\cap L$ is
normalized by $x$ and $P^x \neq P$. Thus  $G_{\alpha \gamma}= P\cap L$ and this group acts on $\Gamma(\gamma)$
preserving the sets $\Gamma(\alpha) \cap \Gamma(\gamma)$ and $\Gamma(\gamma) \setminus (\Gamma(\alpha) \cap
\Gamma(\gamma))$. Because $P\cap L$ is a line stabilizer in $G_{\gamma} = P^x$, it  has orbits of lengths $5$ and
$16$ on $\Gamma(\gamma)$. Since $|G_{\alpha\gamma} :G_{\alpha\beta\gamma}|= |P\cap L:P\cap L\cap B|=5$, we infer that
$\Gamma(\alpha) \cap \Gamma(\gamma)$ has order $5$ or $21$. If the size is $21$, then we have accounted for all
the cosets of $B$ and $P$ in $G$ and we have $22$ cosets of $P$ and $22$ cosets of $B$ which is impossible as $B$
and $P$ have different orders. Thus $|\Gamma(\alpha) \cap \Gamma(\gamma)|=5$.  Let $\theta \in \Gamma(\beta)$
have distance $3$ from $\a$. Then $|\theta^{G_{\alpha\gamma}}|=16$, and $G_{\a\gamma\theta} \cong \Alt(5)$  complements $O_2(G_{\a\gamma})$. Consequently it acts
on $\Gamma(\theta)$ with one fixed point $\gamma$ and an orbit of length $5$. In particular  $G_\a=P$
acts transitively on the set of vertices at distance $3$ from $\a$.  Now consider the path $(\beta,\alpha, \tau)$
where $\tau \in\Gamma(\alpha)\setminus \{\beta\}$.  Then $G_{\beta\a\tau}$ is the intersection of two point
stabilizers in $P$ and has shape   $2^4:3$  and $G_{\b}$ acts transitively on such paths. We make such a path
$(\a,\b ,Bx)$ where $x \in (P\cap L) \setminus B$ and note that the stabilizer of $\a$ and $Bx$ contains $(B\cap
L) \cap (B\cap L)^x \approx  2^4:\Sym(3)$. It follows that $\Gamma(\beta) \cap \Gamma(Bx)$ contains at least $2$
vertices.   The group $2^4:\Sym(3)$ acts on $\Gamma(\beta) $ with orbits of length  $2$ and $4$  and therefore,
since $\Gamma$ is connected, we infer that $|\Gamma(\beta) \cap \Gamma(Bx)|=2$.

In particular, $|\Gamma(\theta) \cap \Gamma(\beta)|\neq 1$. Since $G_{\a\gamma\theta}$ acts on $\Gamma(\theta)$ with an orbit of length $1$ and  an orbit of length $5$, we deduce that every neighbour of $\theta$ is incident to some vertex at distance 2 from $\a$. In particular  $|\mathcal P |\le 22$ and $|\mathcal B|\le 77$. Since $|\mathcal P| \ge 22$,  we have equalities $|\mathcal P|=22$ and $|\mathcal B|=77$. The fact that  $P$ acts two transitively on the 21 points of the projective plane yields that $G$ acts three transitively on $\mathcal P$. In particular, given any three members of $\mathcal P$ we may map them to three neighbours of the coset $B$.

We now identify the members of $\mathcal B$ with their neighbours in $\mathcal P$. Thus $\mathcal B$ becomes a
set of six element subsets of $\mathcal P$ which we call blocks.   Since $G$ acts three transitively on $\mathcal
P$ we get that any three points are contained in a block.   Suppose that $\beta_1$ and $\beta_2$ are blocks
sharing a common point. Then, as we saw earlier, $|\Gamma(\beta_1)\cap \Gamma(\beta_2)|=2$ which means that every
subset of $\mathcal P$ of size $3$ is contained in exactly one block. Thus $(\mathcal P, \mathcal B)$ is a
Steiner triple system with parameters $(3,6,22)$.  Such systems are uniquely determined by \cite{Witt} and
therefore  $G$ is isomorphic to a subgroup of $\Aut(\Mat(22))$. As $G = \langle P,B \rangle$, we see $G =
G^\prime$. So $G \cong \Mat(22)$ and this completes the proof of the lemma.
\end{proof}

The proof of the next lemma is very similar to the previous one and so the proof is somewhat abbreviated.
\begin{lemma}\label{M23r} Suppose that $G$ is a group and $P$, $B$ and $L$ are subgroups of $G$ such
that $P \cong \Mat(22)$,  $B \approx 2^4:\Alt(7)$, $L \approx  2^4:(3 \times \Alt(5)).2$, $B \cap P\approx
2^4:\Alt(6)$, $L \cap P \approx 2^4:\Sym(5)$ and $P\cap B \cap L \approx 2^4:\Sym(4)$.
  Then $\langle P,B,L\rangle \cong \Mat(23)$.
\end{lemma}

\begin{proof}

We again suppose that $G= \langle P,B,L\rangle$ and consider the graph $\Gamma$ which has vertex set   $\mathcal P=\{Pg\mid g \in G\}$ and $\mathcal B = \{Bg\mid g\in G\}$ made into a graph as in Lemma~\ref{M22}. Again we have  $L =  (P\cap L)(L\cap B)$  and that $L$ does not normalize $P$. In particular, we have $\Gamma$ is connected and $G$ acts faithfully on $\Gamma$. We also have that $|\mathcal P| \ge 23$.  For $\a \in \mathcal  P $ and $\b \in \mathcal B$ we know  $|\Gamma(\a)|=77$ and $|\Gamma(\b)|= 7$.

Suppose that $\a = P$ and $\b = B$. Then $G_{\a\b }= P \cap B$ and so $G_{\a\b}$ acts transitively on
$\Gamma(\beta)\setminus \{\a\}$. Let $\gamma \in \Gamma(\beta)\setminus \{\a\}$.  Then $G_{\a\b\gamma}\approx
2^4:\Alt(5)$. Now we let $x \in (B \cap L)\setminus P$  and note that $\gamma=Px \in \Gamma(\beta)$ and that
$G_\a \cap G_\gamma = P \cap P^x \ge (L\cap P)' \cong 2^4:\Alt(5)$. Furthermore we have $G_{\a\b\gamma} \cap
(L\cap P)' \approx 2^4:\Alt(4)$ has index $2$ in $P\cap B \cap L$. We see that $\langle G_{\a\b\gamma},(L\cap
P)'\rangle \cong \PSL_3(4)$ and thus $G_{\a\gamma}\cong \PSL_3(4)$ and, in particular, we have $|\Gamma(\alpha)
\cap \Gamma(\gamma)|= 21$.  Let $\theta \in \Gamma(\gamma)$ have distance $3$ from $\a$ in $\Gamma$. Then as in
$\Mat(22)$ the  stabilizer of a point $p$   has an orbit of length 56 on the blocks  not containing $p$, we see
that $G_{\a\gamma\theta} \cong \Alt(6)$. In particular $G_{\a\gamma\theta}$ has orbits of length $1$ and $6$ on
$\Gamma(\theta)$. Thus we only need to see that $|\Gamma(\theta) \cap \Gamma(\beta)| \ge 3$. We do this exactly
as in the last lemma and in fact we get that $|\Gamma(\b)\cap \Gamma(\theta)|=3$. Thus $|\mathcal P| = 23$ and
$|\mathcal B| = 253$. We now prove that $G \cong \Mat(23)$ just as in Lemma~\ref{M22}.
\end{proof}

We assume that identifications of simple groups by their 2-local structure are fairly well known. This is not be the case for 3-local identifications. So to make the paper self contained we  now state those results which will be used in this paper to identify  groups when  $p=3$ in the main theorem.

 \begin{lemma}\label{F42} Suppose that $G$ is a finite group, $Z \le G$ has order $3$ and set $H = C_G(Z)$. If
the following hold
\begin{enumerate} \item $Q=F^*(H)$ is extraspecial of order $3^{1+4}$ and $Z(F^*(H)) =Z$;
\item $H/Q$ contains a normal subgroup isomorphic to $ \Q_8\times \Q_8$; and \item $Z$ is not weakly closed in a
Sylow $3$-subgroup $S$ of $G$  with respect to $G$, \end{enumerate} then  either $F^*(G) \cong \F_4(2)$ or
$F^*(G) \cong \mathrm{PSU}_6(2)$. \end{lemma}

\begin{proof} See \cite[Theorem 1.3]{F42}. \end{proof}

 \begin{lemma}\label{Co2}  Suppose that $G$ is a finite group, $Z \le G$ has order $3$ and set $H = C_G(Z)$. If the
following hold
\begin{enumerate}
\item $Q=F^*(H)$ is extraspecial of type $3^{1+4}_+$ and $Z(F^*(H)) =Z$;
\item  $F^*(H/Q)= O_2(H/Q)$ is extraspecial of type $2^{1+4}_-$;
\item $H/O_{3,2}(H)\cong \Alt(5)$; and
\item $Z$ is not weakly closed in a
Sylow $3$-subgroup $S$ of $G$  with respect to $G$,
\end{enumerate}
 then $G \cong \Co_2$.
 \end{lemma}

 \begin{proof} This is the main theorem of \cite{Co2}. \end{proof}

  \begin{lemma}\label{McL}  Suppose that $G$ is a finite group, $Z \le G$ has order $3$ and set $H = C_G(Z)$. Let further $S$ be a Sylow $3$-subgroup of $H$ and $J$ be some elementary abelian subgroup of $S$ of order $3^4$. If the
following hold
\begin{enumerate}
\item $Q=F^*(H)$ is extraspecial of type $3^{1+4}_+$ and $Z(F^*(H)) =Z$;
\item  $F^*(H/Q) \cong 2\udot\Alt(5)$; and
\item  $J = F^\ast(N_G(J))$ and $O^{3^\prime}(N_G(J)/J) \cong \Alt(6)$.
\end{enumerate}
 then $F^\ast(G) \cong \McL$.
 \end{lemma}

  \begin{proof} This is from  \cite{McL}. \end{proof}

\begin{lemma}\label{2E6}  Suppose that $G$ is a finite group, $Z \le G$ has order $3$ and set $H = C_G(Z)$. If the
following hold
\begin{enumerate} \item $Q=F^*(H)$ is extraspecial of order $3^{1+6}$ and $Z(F^*(H)) =Z$; and
\item $O_{2}(H/Q) \cong \Q_8\times \Q_8\times \Q_8$;
\item $Z$ is not weakly closed in a
Sylow $3$-subgroup $S$ of $G$  with respect to $G$, \end{enumerate}
then $F^*(G) \cong {}^2\E_6(2)$. \end{lemma}

\begin{proof} This is \cite[Theorem 1.3]{M22}. \end{proof}

\begin{lemma}\label{M221}Suppose that $G$ is a finite group, $Z \le G$ has order $3$ and set $H = C_G(Z)$. If the
following hold
\begin{enumerate} \item $Q=F^*(H)$ is extraspecial of order $3^{1+6}$ and $Z(F^*(H)) =Z$;  \item
$O_{2}(H/Q)$ acts on $Q/Z$ as a subgroup of order $2^7$ of $\Q_8\times \Q_8\times \Q_8$, which contains
$Z(\Q_8\times \Q_8\times \Q_8)$; and
\item $Z$ is not weakly closed in a
Sylow $3$-subgroup $S$ of $G$  with respect to $G$,\end{enumerate}
then $F^*(G) \cong \M(22)$.\end{lemma}

\begin{proof} This is \cite[Theorem 1.4]{M22}. \end{proof}

\begin{lemma}\label{M231} Suppose that $G$ is a finite group, $Z \le G$ has order $3$ and set $H = C_G(Z)$. If the
following hold
\begin{enumerate}
\item $Q=F^*(H)$ is extraspecial of type $3^{1+8}_+$ and $Z(F^*(H)) =Z$;
\item  $F^*(H/Q)= O_2(H/Q)$ is extraspecial of type $2^{1+6}_-$;
\item $H/O_{3,2}(H)$ is isomorphic to the centralizer of a $3$-central element in $\PSp_4(3)\cong \Omega_6^-(2)$; and
\item $Z$ is not weakly closed in a
Sylow $3$-subgroup $S$ of $G$  with respect to $G$,
\end{enumerate}
 then
$G$ is isomorphic to $\M(23)$.
\end{lemma}

\begin{proof} This is \cite[Theorem 1.3]{PaS}. \end{proof}

\begin{lemma}\label{F2}Suppose that $G$ is a finite group, $Z \le G$ has order $3$ and set $H = C_G(Z)$. If the
following hold
\begin{enumerate}
\item $Q=F^*(H)$ is extraspecial of type $3^{1+8}_+$ and $Z(F^*(H)) =Z$;
\item  $F^*(H/Q)= O_2(H/Q)$ is extraspecial of type $2^{1+6}_-$;
\item $H/O_{3,2}(H)\cong \Omega_6^-(2)$; and
\item $Z$ is not weakly closed in a
Sylow $3$-subgroup $S$ of $G$  with respect to $G$,
\end{enumerate}
then $G \cong \F_2$.
 \end{lemma}

\begin{proof} This is \cite[Theorem 1.4]{PaS}. \end{proof}

\section{The configurations with  $p=2$}

We first of all establish some notation that will be used in this section and in Section~4. Assume that $G$ and $H$ are as in the statement of the main theorem. Thus $F^\ast(H)$ is a simple group of Lie type  in characteristic $p$ and of rank at least $2$. We
set $H^*= F^*(H)$,  let $S$ be a Sylow $p$-subgroup of $H$, $S^* = S \cap H^*$ and $z$ be a root element in
$Z(S^*)$. Throughout we
assume that $Q \le S$ is a large subgroup and that $N_G(Q) > N_H(Q)$. Note that if $z$ centralizes $Q$ then
$C_H(z)$ normalizes $Q$ and so $Q \le O_p(C_H(z))$.

In this section we assume in addition that $p=2$. We work through the various possibilities for $H^*$ provided by Lemma~\ref{possible}.


\begin{lemma}\label{Q} If $H^* \cong \Omega^+_n(2)$ for $n = 6,8$, then $Q = O_2(C_{H^*}(z))$ and is extraspecial. \end{lemma}

\begin{proof} Note that $z \in Z(Q)$ and so $Q \le O_2(C_H(z))$. If $n=8$, then $C_H(z)$   is a maximal subgroup of $H$,
  $O_2(C_H(z))$ is extraspecial and $C_H(z)/O_2(C_H(z))$
acts irreducibly on $O_2(C_H(z))/\langle z\rangle$. Hence we have $Q= O_2(C_G(z))$ and the result is true in this case.

So assume that $n = 6$. Then $H^*\cong \SL_4(2)$ and $C_{H^*}(z)\approx 2^{1+4}_+.\SL_2(2)$ is not a maximal subgroup of $H^*$.  We explore the possibilities for $Q$ under the assumption that it is not $O_2(C_{H^*}(z))$.

The structure of
$C_{H^*}(z)$
yields  that in $O_2(C_{H^*}(z))$ there are exactly three non-central proper normal subgroup of $C_{H^*}(z)$. They
each have order $2^3$ and are abelian. Two of these group are normal in parabolic subgroups of $H^*$
which have Levi factors isomorphic to $\SL_3(2)$. These subgroups cannot be candidates for $Q$ as otherwise $N_G(Q)=N_H(Q)$.

 Let $R$ be the third normal subgroup of
$O_2(C_{H^*}(z))$ of order $2^3$ which is normalized by $C_{H^*}(z)$ and assume that $Q \cap O_2(C_{H^*}(z)) = R$.

Assume that $Q$ is not abelian. Then, as $Q$ is normalized by $C_{H^*}(z)$ and $C_{H^*}(z)$ acts irreducibly on $R/\langle z\rangle$,  $R$ has index $2$ in $Q$ and $Z(Q)=\Phi(Q)=Q'= \langle z\rangle$. But then $Q$ is extraspecial of order $2^4$ which is ridiculous. Thus $Q$ is abelian.
Since $R$ contains an element  which is not a transvection,  after
identifying $\SL_4(2)$ with $\Alt(8)$, we may assume that  $(1,2)(3,4) \in R$. But $N_{H^*}(R) = C_{H^*}(z)$ and this group does not contain
$C_{\Alt(8)}((1,2)(3,4))\approx (2^2 \times \Alt(4)).2$. Hence $Q \cap O_2(C_{H^*}(z)) \neq R$.   Since $Q\cap O_2(C_{H^*}(z)) \neq \langle z\rangle$, we now know that $Q > O_2(C_{H^*}(z))$. Therefore  $H
\cong \Sym(8)$ and $Q=O_2(C_H(z)) > O_2(C_{H^*}(z))$. Set
$Q_1= O_2(C_{H^*}(z))$. Then $[Q, Q_1]$ is normalized by $C_H(z)$ and consequently has order $2^3$. It follows
that $Q_1$ is the preimage of the Thompson subgroup of $Q/\langle z \rangle$ and therefore $Q_1$ is a
characteristic subgroup of $Q$. In particular, $N_G(Q) \le N_G(Q_1)$ and $N_G(Q)$ acts on $Q_1$ and centralizes
$Q/Q_1$. Thus all the elements of odd order in $N_G(Q)$ which centralize $Q_1$ also centralize $Q$. So we see
that $O^2(N_G(Q)/Q_1)$ acts faithfully on $Q_1$. But then $N_G(Q_1)/Q_1$ embeds into $\mathrm O_4^+(2) \cong
\Sym(3)\wr 2$ as $Q_1$ is extraspecial of order $2^5$ and $+$-type. As $Q/Q_1$ is a subgroup of order two
which is centralized by $N_G(Q)$, we get that $N_G(Q) \leq H$, which is a contradiction. Thus $Q=
O_2(C_{H^*}(z))$ and is extraspecial as claimed.
\end{proof}

\begin{lemma}\label{O6} If $H^* \cong \Omega^+_n(2)$ for $n \in\{ 6,8\}$, then $n=8$ and $H^*\cong \Omega^+_8(2)$.
\end{lemma}

\begin{proof}
We suppose   that $H^* \cong \Omega^+_6(2)$ and aim for a contradiction. By Lemma \ref{Q}
$Q=O_2(C_{H^*}(z))$ is extraspecial of order $2^5$. Hence  $N_G(Q)/Q$ is isomorphic to a subgroup of
$\OO^+_4(2)$ {and $|N_G(Q):N_H(Q)|=3$.} Let $\rho \in C_H(z)$ be of order three and $\nu \in N_G(Q)$ with
$\langle \nu, \rho \rangle$ elementary abelian of order 9 and $\nu$ and $\rho$ conjugate in the $\OO_4^+(2)$. Let $x \in N_{H^*}(Q) \setminus Q$ be a 2-element.
Then $x$ normalizes the three elementary abelian subgroups of order 8 in $Q$ on which $\rho$ acts {(note that
$\rho$ also normalizes two quaternion subgroups of order $8$).} Furthermore $\nu$ permutes the three
$\rho$-invariant elementary abelian subgroups of order $8$ transitively. This gives that $[\nu,x] \in Q$  and
hence $\nu$ normalizes $S^*=Q\langle x\rangle$. As the Thompson subgroup $E$ of $Q\langle x \rangle$ is
elementary abelian of order 16,  $\nu \in N_G(E)$. Furthermore  $N_H(E)$ induces
{either}  $\Sym(3) \times \Sym(3)$ or $\Sym(3) \wr  2$ on $E$. As $N_H(E)$ contains a Sylow 2-subgroup of
$N_G(E)$,  $N_H(E)$ acts irreducibly on $E$ and $N_G(E)/C_G(E)$ is isomorphic to a subgroup of $\SL_4(2)$, we easily see that
$N_G(E) = C_G(E)N_H(E)$. Since $G$ is of parabolic characteristic $2$, we also have $C_G(E) = E$. But then $$\nu \in N_G(E) = C_H(E) \le
H,$$ which is  a contradiction. We conclude that $N_G(Q) \leq H$ which is impossible. Hence $n \neq 6$.
\end{proof}

\begin{proposition}\label{O8} If $H^* \cong \Omega^+_8(2)$, then $F^\ast(G) \cong \mathrm P\Omega^+_8(3)$.
\end{proposition}

\begin{proof} Suppose that $H^* \cong \Omega^+_8(2)$. Then  $C_{H^*}(z)/O_2(C_H(z)) \cong \SL_2(2)\times \SL_2(2)\times \SL_2(2)$,  $Q=O_2(C_H(z))=O_2(C_{H^*}(z))$ is extraspecial of order $2^9$ by
Lemma~\ref{Q} and $|S:Q|\le 2^4$ as $\Out(H^*)\cong \Sym(3)$. {We know that $Z_2(S) =Z_2(S^*)$ has order $4$ and  is normalized by the parabolic subgroup $P$ of $H^*$ corresponding to the middle node of the Dynkin diagram of $G$ and we have $$P = \langle Q^g \mid g \in G,  z^g \in Z_2(S)\rangle S^*$$ is normalized by $N_G(S^*)$. Thus $N_{N_G(N_{H^*}(Q))}(S^*)$ normalizes $N_{H^*}(Q)$ and $P$. Since $H^*=\langle P, N_{H^*}(Q)\rangle$,
  $N_{N_G(N_{H^*}(Q))}(S^*)\le H$. In particular, $N_G(Q)$ does not normalize $N_{H^*}(Q)$ for otherwise $$N_G(Q)= N_{H^*}(Q)N_{{N_G(Q)}}(S^*) \le H.$$}

Let $\ov {X}= N_G(Q)/Q$ and $\ov Y= O_3(\ov {N_H(Q)})$. Then $\ov {N_{H^*}(Q)} =\ov {S^*Y}$.

We have that $\ov S^*$ is elementary abelian and by Lemma~\ref{O8mod} (i), (ii) and (iii) there is a unique element $e$ in  $\ov S^*$ such that $[Q/\langle z \rangle,e]$ is not totally singular. Thus $e$ is weakly closed in $\ov {S^*}$ with respect to $\ov X$. We furthermore note that $e$ inverts  $\ov Y$. Suppose that $S> S^*$. Then  $\ov S \cong 2 \times \Dih(8)$.  Lemma~\ref{O8mod} (iv) yields an involution $f \in \ov S \setminus \ov {S^*}$ such that $|[Q/\langle z \rangle,f] |=4$ and so $f$ is not $\ov{X}$-conjugate to any element of $\ov{S^*}$ as all the non-trivial elements in this group have commutator of order $2^4$ on $Q/\langle z \rangle$.  Let $\ov F$ be the elementary abelian group of order $8$ in $\ov S$ with $\ov F \not = \ov {S^*}$.   Then there are at most two elements (and they are conjugate) in $\ov F \setminus \ov S^*$ which could be $\ov X$-conjugate to an element of $\ov {S^*} \cap \ov F$. It follows that the other two elements in $\ov {S^*} \cap \ov F$ are not fused to any element of $\ov F$. Thus $N_{{\ov{X}}}(\ov F)$ normalizes $\ov {S^*}\cap \ov F$ and $$N_{\ov{X}}(\ov F)= N_{\ov{X}}(\ov {S^*} \cap \ov F)S.$$

If $e$ is not weakly closed in $\ov S$ with respect $\ov {X}$, then there is a $\ov{X}$-conjugate $e^x$ of $e$ in $\ov F \setminus \ov {S^*}$. Hence  $\ov F^x $ and $\ov F$ are both contained in the centralizer of $e^x$ and, as $\ov F$ is not $\ov X$-conjugate to $\ov S^*$, we conclude that there exists $y \in C_{\ov{X}}(e^x)$ such that $\ov{F^{xy}}= \ov F$. Thus $e^x$ and $e$ are conjugate in $N_{\ov X}(\ov F)= N_{\ov{X}}(\ov {S^*} \cap \ov F)S$ which is impossible. Therefore  $e$ is weakly closed in $\ov S$ with respect to $\ov X$.
Application of the $Z^\ast$-theorem \cite{Glau} now shows that $${e} \in Z^\ast(\ov{X}).$$
 In particular  $$\ov{Y}= [\ov Y,e]\leq O_{2'}(\ov {X}).$$

 Suppose that  $\ov{Y}=O_{2'}(\ov X)$. Then as, by Lemma~\ref{O8mod} (v),  $\ov Y$ contains exactly six elements which act fixed point freely on $Q/\langle z \rangle$ and these elements form three blocks of size two, we have that $\ov X/\ov Y$ embeds into $2\wr \Sym(3)$. Since when $S>S^*$, we have $N_{\ov{X}}(\ov F)= N_{\ov{X}}(\ov {S^*} \cap \ov F)S$ we see that in any case $\ov{S^*Y}$ is normal in $\ov X$. But then the  Frattini Argument gives $\ov X = \ov Y N_{\ov X}(S^*) \le H$ which is a contradiction. Therefore $\ov Y < O_{2'}(\ov X)$. The structure of $\Out(Q)\cong \mathrm O_8^+(2)$, now shows that $O_{2'}(\bar{X})$ is elementary abelian of order 81. As there is just one class of elementary abelian groups of order 81 in $\Out(Q)$ (the Sylow $3$-subgroup is contained in $\mathrm O_2^-(2) \wr \Sym(4)$),  the action of $O_{2'}(\ov X)$ on $Q$ is uniquely determined.
 So the preimage of $O_{2'}(\ov X)$ is a central product of four groups $K_1, K_2, K_3$ and $K_4$ isomorphic to $\SL_2(3)$ and these are the only subgroups of $C_{H^*}(z) \cong \SL_2(3)$.
 In particular,  $N_G(Q)$ contains
a subnormal subgroup    isomorphic to  $\SL_2(3)$. Now if $S^*$ normalizes $K_i$, $1\le i\le 4$, then each $a \in
S^*$ would have elements of order $4$ in their commutator contrary to Lemma~\ref{O8mod}(i). Hence we see that  there is an
$i\in \{1,2,3, 4\}$ and $s \in S^*$ such that $[K_i^s,K_i]=1$. Therefore an application of Aschbacher's Classical Involution
Theorem \cite[Theorems 7 and 8]{Asch1} yields  $F^\ast(G)$ is either $\mathrm P \Omega_n^\pm(3)$ for $5\le
n\le 8$ or $\Sp_6(2)$.  As $|S^*| = 2^{12}$, we see that $ F^\ast(G) \cong\mathrm P
\Omega^\pm_8(3)$. However,   $\mathrm P
\Omega^-_8(3)$  does not contain a subgroup isomorphic to $\Omega_8^+(2)$ (as $5^2$ divides the order of one but not the other). Hence we have  $F^\ast(G) \cong \mathrm P\Omega^+_8(3)$.\end{proof}

%
%
%
%

\begin{lemma}\label{U42} We have that $H^* \not\cong \U_4(2)$.
\end{lemma}

\begin{proof}  Suppose that $H^*\cong \U_4(2)$. Then $C_H(z)$ is a maximal subgroup of $H$ and $C_H(z)$ acts
irreducibly on $O_2(C_H(z))/\langle z\rangle$. This means that $N_H(Q)= C_H(z)$ and $Q=O_2(C_H(z))$. Thus $Q$ is
extraspecial of order $2^5$ and has outer automorphism group $\mathrm O_4^+(2)\cong \Sym(3)\wr 2$. Since
$|C_H(z)/O_2(C_H(z))|$ is divisible by $9$,  $O_{2,3}(N_G(Q)) \leq H$ and so $N_G(Q) \leq H$. This
proves the lemma.
\end{proof}

{\begin{lemma}\label{U52} We have that $H^* \not\cong \U_5(2)$.
\end{lemma}

\begin{proof} Suppose that $H^*\cong \U_5(2)$. Then again $C_{H}(z)= N_H(Q)$ and this time $Q$ is extraspecial of order $2^7$ with outer
automorphism group $\mathrm O_6^-(2)$. Since  $N_G(Q)\neq N_H(Q)$ and, by \cite{Atlas},  $N_H(Q)/Q \cong
\GU_3(2)$ is a maximal subgroup of $\mathrm \Omega_6^-(2)$, we have $N_G(Q)/Q$ contains a subgroup isomorphic to
$\Omega_6^-(2)$.This is ridiculous as $|S:Q|=|N_H(Q)/Q|_2\le 2^4$ and therefore the lemma is true.
\end{proof}
}

\begin{proposition}\label{G2} If $H^* \cong \G_2(2)^\prime$, then $G \cong \G_2(3)$.
\end{proposition}

\begin{proof} In this case  $C_H(z)$ is a maximal subgroup of $H$ and so $C_H(z)= N_H(Q)$.
We first show that $H \cong \G_2(2)$. Otherwise we have $H \cong \U_3(3) \cong \G_2(2)^\prime$ and then
$O_2(C_H(Z(Q))) \cong 4 \ast \Q_8$. But then $O_{2,3}(N_G(Q)) \leq H$ and so $N_G(Q) \leq H$, a contradiction. So $H \cong \G_2(2)$.

 Now $Q \le O_2(C_H(z))$ is normal in $C_H(z)$. If $Q \neq O_2(C_H(z))$, then, as the element $\tau$  of order $3$ in $C_{H}(z)$ has $[O_2(C_H(z)),\tau]\cong \Q_8$, we have $[O_2(C_H(z)),\tau]\le Q < O_2(C_H(z))$. It follows that $Q \cong \Q_8$ or $4\ast \Q_8$ and then we have the same contradiction as in the last paragraph. Hence $Q = O_2(C_H(z))\cong 2^{1+4}_+$ and the outer automorphism group of $Q$ is
$\mathrm O_4^+(2) \cong \Sym(3)\wr 2$.
 As $N_G(Q) \not\leq H$, we now get that $|N_G(Q) : N_H(Q)| = 3$ and so
$N_G(Q)$ contains a normal subgroup of index two isomorphic to $\SL_2(3) \ast \SL_2(3)$. We now consider the
other  parabolic subgroup $P$ of $H$ containing $S$. We have that $P$ has shape $((4\times 4):2) .\Sym(3)$, where
the homocyclic subgroup of shape $4\times 4$ is inverted in $O_2(P)$. {Let $x \in P\setminus C_H(z)$ and consider the
subgroup $E=Q\cap Q^x$. Since $z\neq z^x$ and $\langle z^x, z\rangle$ has order $4$ and is contained in $Q$, we
get that $E$ is elementary abelian and, as $|Q:Q\cap O_2(P)|= 2$ and $|(Q \cap O_2(P))Q^x/Q^x|\le 2$, we infer
that $E$ has order $8$. Additionally, as $P$ has two non-central chief factors in $O_2(P)$, we have $P/E\cong
\Sym(4)$.}

If $G$ has a subgroup $G_1$ of index $2$, then, as $N_G(Q)$ does not normalize a subgroup of index two in $Q$, we
get that $Q \le G_1$ and of course $H^* \le G_1$. But then  $H = H^*Q \le G_1$, a contradiction, as $H$ contains
a Sylow 2-subgroup of $G$.   Thus $G$ does not have a  subgroup of index 2. In particular, by the Thompson
Transfer Lemma \cite[Lemma 15.16]{GLS2}, all the involutions of $H$ are conjugate to involutions in $H^*$. As
$H^*$  has a  unique conjugacy class of involutions, the same is true for $G$. Therefore all the involutions in
$E$ are $G$-conjugate. Let $t \in E\setminus H^*$. Then $t^P$ has order $4$ and $C_P(t)E/E \cong \Sym(3)$.
Especially we have $E= \langle t\rangle [E,C_P(t)]$. It follows that $E \le C_G(t)'$. As $C_G(t)/O_2(C_G(t))\cong
N_G(Q)/Q$  has Sylow $2$-subgroups of order $2$, we have $E \leq O_2(C_G(t))$. In particular, there are elements
of $N_G(E)$ which induce transvections on $E$ with commutator $\langle t\rangle$. Using  $P/E \cong \Sym(4)$, we
now have $N_G(E)/C_G(E) \cong \SL_3(2)$. Since $C_G(E)= C_{N_G(Q)}(E)=E$, we now have  $N_G(E)/E \cong \SL_3(2)$.
Finally applying  \cite{Asch} we get that $G \cong \G_2(3)$.
\end{proof}

{\begin{lemma}\label{sp62} We have $H^*\not\cong \Sp_6(2)$.\end{lemma}

\begin{proof} Since $H^*= H$, we have that $Q \ge Z(S)$ and then
$Q$ is normalized by $\langle C_H(z) \mid z \in Z(S^\#)\rangle = H$. This shows that $H^* \not \cong \Sp_6(2)$.
\end{proof}}

\begin{lemma}\label{twistedF4} We have $H^*\not\cong {}^2\F_4(2)^\prime$. \end{lemma}

\begin{proof} Suppose that $H^*\cong {}^2\F_4(2)^\prime$. Then,  by Lemma~\ref{2f4p} (v) and (vi), $ C_{H^*}(z)$ has
shape $2^{1+4+ 4}.{}^2\B_2(2)$. Set $R= O_2(C_H(z))$ and $R^* = R \cap H^*$. Suppose that $Q \not \ge R^*$. Then,
as $Q$ admits ${}^2\B_2(2)$ non-trivially Lemma~\ref{2f4p} (vii) implies that $Q \cap R^* = Z_2(R^*)$ is elementary
abelian. If $Q > Q\cap R^*$, we further infer from Lemma~\ref{2f4p}(viii) that $|Q|= 2^6$,  $Q$ is abelian but not elementary abelian and
that $Q\cap R^*=\Omega_1(Q)$. Hence $Q \cap R^*$ is normalized by $N_G(Q)$ and $N_G(Q\cap
R^*)/C_G(Q\cap R^*)$ embeds into $\GL_5(2)$ and contains a subgroup  $N_{H^*}(Q)/C_G(Q\cap R^*) \approx
2^4.{}^2\B_2(2)$. Since $[O_2(C_{H^*}(z)),Q\cap R^*]= \langle z \rangle$, we see that $N_{H^*}(Q)$ contains all the
transvections to the point $\langle z\rangle$. As $C_{H^*}(z)$ acts irreducibly on $O_2(C_{H^*}(z))/(Q\cap R^*)$,
it follows that either $N_G(Q)= N_H(Q)$ or $N_G(Q)/C_S(Q\cap R^*)$ is isomorphic to $\GL_5(2)$. The first
is impossible by assumption and the second implies that $2^{12}\ge |S|\ge 2^{15}$  which is absurd. Hence $Q \ge O_2(C_{H^*}(z))$. Suppose that $Q= O_2(C_{H^*}(z))$ and $H> H^*$. Then, by Lemma~\ref{2f4p}
(vi) and (vii), $Q/Z_2(Q)$, $Z_2(Q)/Z(Q)$ and $Z(Q)$ are irreducible $N_G(Q)/Q$-sections which are all centralized
by $O_2(C_H(z))$ and this is impossible. Thus $Q= O_2(C_H(z))$.

In particular, $N_G(Q)/Q$ has cyclic Sylow $2$-subgroups and consequently $N_G(Q)/Q$ has a normal $2$-complement.
Since $Z(Q)= \langle z \rangle$ and $Z_2(Q)$ is elementary abelian of order $2^5$,  $N_G(Q)/C_G(Z_2(Q))$ embeds
in to the parabolic subgroup of shape $2^4:\SL_4(2)$ in $\SL_5(2)$. As $Q/C_Q(Z_2(Q))$ is elementary abelian of
order $2^4$, we now have $N_G(Q)/Q$ is isomorphic to a subgroup of $\SL_4(2)$ and therefore $O^2(N_G(Q)/Q)$ must
be a cyclic group of order $15$ and furthermore the centralizer of the element of order $3$  in $N_G(Q)/Q$ is
isomorphic to $3 \times \Dih(10)$. Let $\tau$ be an element of order $3$ in $N_G(Q)$, $\beta$ be an element of
order $5$ which commutes with $\tau$ and $t\in S$ be a $2$-element that centralizes $\tau$ and is not contained
in $Q$. Note that as $\tau$ is contained in a cyclic group of order $15$, $C_{O_2(C_{H^*}(z))}(\tau) =
C_{O_2(C_{H^*}(z))}(\beta) $ has order $2$. Set $L= \langle \tau ,S\rangle$.  Further assume that $M$ is the
parabolic subgroup of $H$ containing $S$ with $M \neq N_H(z)$. Set $K= \langle L,M\rangle$.

As $Z_2(Q)/Z(Q)$ is the $4$-dimensional ${}^2 \B_2(2)$-module and as $t$ inverts $\beta$ mod $Q$,
$C_{Z_2(Q)/Z(Q)}(t)$ has order $2^2$ and is normalized by $S$. Let  $W$ be the preimage of $C_{Z_2(Q)/Z(Q)}(t)$.
Then $W$ is elementary abelian of order $2^3$ and  $\tau$ acts non-trivially on $W/Z(Q)$. Furthermore $L$
normalizes $W$ and $W$ corresponds to $Z_3(T)$ in Lemma~\ref{2f4p}. Hence exploiting Lemma~\ref{2f4p} (iii) again
we see that  $W$ is normalized by $M$ and that  $M$ acts as $\Sym(4)$ on $W$. Thus $K$ normalizes $W$ and
$K/C_K(W) \cong \SL_3(2)$. Observe that $C_K(W)$ centralizes $Z(Q)$ and hence normalizes $Q$ as $Q$ is
large. Since $N_G(Q)/Q \approx 15\cdot 4$, we infer that $C_K(W) = C_Q(W)$ has order $2^8$ or $2^9$. Moreover, we
remark that $C_{C_Q(W)}(\tau)$ has order at most $8$.  If $|C_K(W)| = 2^8$, then there are at most two
3-dimensional $K/C_K(W)$-modules involved in $C_Q(W)$. Thus, in this case,  $|C_{C_K(W)}(\tau)| \geq 2^{1+1+2} =
16$, which is a contradiction. So we have $|C_K(W)| = 2^9$ and  all the chief-factors for $K$ in $C_K(W)$ are
3-dimensional $K/C_K(W)$-modules. In particular $H > H^*$.  By Lemma \ref{2f4p}(iv), $|Z(S/W)| = 2$.
Hence $C_K(W)/W$ is a non-split extension of two 3-dimensional modules. Choose $x \in Z_2(Q)$ so that $x$
projects to a non-trivial element in $C_{Z_2(Q)/W}(S)$ and set $U = \langle x^K\rangle W$. We have that $U/W$ is
 of order $2^3$.  As $K$ acts transitively on the vectors in the natural
$3$-dimensional $\SL_3(2)$-module and $x$ has order $2$, we have that $U$ has exponent 2 and so is elementary
abelian. This contradicts Lemma~\ref{2f4p} (i) and proves the lemma.

\end{proof}

{\begin{proposition}\label{q=2sp4} If $H^*\cong \Sp_4(2)'$, then $G \cong \Mat(11)$.
\end{proposition}

\begin{proof} If $H= H^* \cong \Sp_4(2)'= \Alt(6)$, then $G$ has dihedral Sylow $2$-subgroups of order $8$. The
candidates for $Q$ are the two elementary abelian groups of order $4$, the cyclic group of order $4$ and the
Sylow $2$-subgroup itself. In each case we have $N_G(Q)\le H$ and we are done.

Suppose that $|H:H^*|=2$. Then $S \cong 2 \times \Dih(8)$, $\Dih(16)$ or $\SDih(16)$ corresponding to $H \cong
\Sym(6)$, $\PGL_2(9)$ and $\Mat(10)$ respectively. Consider the first case. Then $|Z(S)| = 4$ and $Z(S) \le
Z(Q)$. But $H = \langle C_H(z)~|~ z \in Z(S)^\sharp \rangle \le N_G(Q)$, which is impossible.

If $S$ is a dihedral group of order $16$, then all the candidates for $Q$  have $2$-groups as their automorphism
group. So again $N_G(Q)\le H$, a contradiction. So suppose that $S$ is semidihedral of order $16$. Then the
possibilities for $Q$ are a quaternion group of order $8$, a cyclic subgroup of order $8$, a dihedral subgroup of
order $8$ and $S$ itself. Since the only one of these groups with a non-trivial automorphism of odd order is the
quaternion group, we infer that $C_G(z) \cong \GL_2(3)$ and consequently $G \cong \Mat(11)$ or $\PSL_3(3)$ by
\cite{Brauer}. Since $5$  divides the order of $H$ but not the order of $\PSL_3(3)$, we conclude that $G \cong
\Mat(11)$.

Suppose that $|H:H^*|= 4$.  Let $M$ be the subgroup of $H$ with $M \cong \Mat(10)$. Then all the involutions in
$M$ are contained in $H^*$. Let $t \in H\setminus M$ be an involution such that $C_S(t) \cong \Dih(8) \times 2$
(so $t$ is in the subgroup of $H^*$ isomorphic to $\Sym(6)$).  Since $Z(S)$ has order $2$, we see that it is
impossible for $t$ to be conjugate to a $2$-central involution and therefore $G$ has a subgroup $G^*$ of index
$2$ with $t \not\in G^*$ by the Thompson Transfer Lemma \cite[Theorem 15.16]{GLS2}. In particular $G^*$ has
dihedral or semidihedral Sylow 2-subgroups.  If $Q \cap G^* $ is a large subgroup of $G^*$, we may apply
induction and get that $G^* \cong \Mat(11)$ and obtain a contradiction as $\Out(\Mat(11))$ is trivial. Thus $Q
\cap G^*$ is not a large subgroup of $G^*$. Certainly $Q \cap G^*$ is normalized by $N_{G}(Q) \cap G^*$ and
$O^2(N_G(Q)) \le G^*$. It follows   that $Q\cap G^*$ is a quaternion group of order $8$ and $|Z(Q \cap G^*)| =
2$.   We now have that $C_G(Z(Q\cap G^*))$
 normalizes $Q\cap G^*$. In addition, $C_{G^*}(Q\cap G^*) \le N_G(Q)$ and
 so $Q \cap G^*$ is a large subgroup of $G^*$ and we have a contradiction.
 \end{proof}}

\begin{lemma} \label{Sp4} We have that $H^* \not \cong \Sp_4(2^a)$ with $a \ge 2$.
 \end{lemma}
\begin{proof}   Suppose that in fact $H^* \cong \Sp_4(2^a)$ for some $a \ge 2$.

Let $P_1$ and $P_2$ denote the maximal parabolic subgroups of $H^*$ containing $S^*$ and define $E_i= O_2(P_i)$
for $i=1,2$. Thus $E_1$ and $E_2$ are elementary abelian $2$-subgroups of $S^*$ of order $2^{3a}$, $S^*=E_1E_2$,
$E_1\cap E_2 = Z(S^*)$ and every element of order $2$ in $S^*$ is contained in $E_1 \cup E_2$. Furthermore, for
$i=1,2$, we have $O^{2'}(P_i)/E_i \cong \SL_2(2^a)$ and $E_i$ is the $\Omega_3(2^a)$-module for $O^2(P_i)$.

 If
$Z(Q)$ contains a root elements of $H^*$, then we may suppose that $N_G(Q)$ contains $P_1$ from which it follows
that $Q=E_1$. But then $Q$ contain both long and short root elements and thus $N_G(Q)$ also contains $P_2$ which
is a contradiction as $H=\langle P_1,P_2\rangle$.  It follows that $H > H^*$ and that $Q$ contains an element
which acts as a graph automorphism of $H^*$. Furthermore, $Z(Q)$ contains no elements which are conjugate to root
elements of $H^*$.

Since $Q$ is normal in $S$ and contains an element $\alpha$ with $E_1^\alpha = E_2$ we have that $(Q\cap S^*)E_1
= [E_1,\alpha]E_1 = S^*$. Therefore $$Z(S^*)=[E_1,S^*]= [E_1,Q\cap S^*] \le Q$$ as $E_1$ is the orthogonal module
for $O^{2'}(P_1)/E_1$ and $a\neq 1$. Thus we have $Z(S^*)=E_1 \cap E_2 \le Q$. We claim that $E_1 \cap Q$ and
$E_2\cap Q$ are the unique elementary abelian subgroups of maximal rank in $Q$ (they may be equal). Suppose that
$E$ is a further elementary abelian subgroup of $Q$. We have $E\cap S^* \le E_1$ or $E\cap S^*\le E_2$ as all the
involutions of $S^*$ are contained in $E_1\cup E_2$. So  we may assume that $E$ is not contained in $S^*$. Then
$C_{E_1\cap E_2}(E)$ has index  $2^a$ in $E_1\cap E_2$, and so we see that $E\cap S^*$ has index at least $4$ in
$E_1\cap Q$ or in $E_2\cap Q$ and this proves our claim since then $|E|<|E_1\cap Q|$. We now have that $E_1 \cap
E_2 = E_1\cap Q\cap E_2 \cap Q$ is a characteristic subgroup of $Q$ and thus we have
$$N_G(Q) \le N_G(E_1\cap E_2).$$

 Since $E_1\cap E_2$ is normalized by $S$ and $G$ is of parabolic characteristic $2$, we have $$F^*(N_G(E_1\cap E_2))=
O_2(N_G(E_1\cap E_2)).$$ Now $S^* $ is a Sylow $2$-subgroup of $C_G(E_1\cap E_2)$ and as $S^*/(E_1\cap E_2)$ is
abelian, we have $S^* = F^*(C_G(E_1\cap E_2))=O_2(C_G(E_1\cap E_2))$. In particular, $S^*$ is normalized by
$N_G(Q) $ and so $S^* \le Q =O_2(N_G(Q))$ and we conclude that $Q=S$ as $S/S^*$ is a cyclic group which is
generated by the graph automorphism \cite[Theorem 2.5.12]{GLS3}.

Since $N_G(Q)= N_G(S)$, $N_G(Q)$ permutes $E_1$ and $E_2$,  $N_G(Q)$ has a subgroup $N_0$ of index $2$ which
normalizes both $E_1$ and $E_2$ and furthermore $N_G(Q)=N_{N_G(Q)}(E_1)Q$. Let $Q_0 = Q \cap N_0$. Then $Q_0$ is
normal in $N_G(Q)$.

We claim that $N_0$ normalizes $H^*$.  We have $E_1\cap E_2$ contains two root subgroups $R_1$ and $R_2$ and no
involution of $(E_1\cap E_2)\setminus (R_1\cup R_2)$ is conjugate into either $R_1$ or $R_2$,  as such
involutions are   conjugate to involutions in $Z(Q)$ and the involutions in $R_1$ and $R_2$ are not. Since $N_0$
normalizes $E_1\cap E_2$, $N_0$ also normalizes $R_1$ and $R_2$ (as $S \cap N_0$ does). It follows that
$M=\langle N_0,P_1\rangle $ acts on $R_1$ and we get $C_{M}(R_1)$ is a normal subgroup of $M$ which contains
$O^{2'}(P_1)$. Suppose that $C_M(E_1) \neq E_1$. Then $C_M(E_1)\cap S \le C_S(E_1) = E_1$ and so we infer that
$C_M(E_1) = E_1J$ where $J$ is a group of odd order. Since $J$ normalizes $S^*= F^*(C_G(E_1\cap E_2))$ which has
class $2$ and since $E_1$ is a maximal abelian subgroup of $S^*$, $J$ centralizes $S^*$ by Lemma~\ref{w=1}.
Therefore $J=1$ and $C_G(E_1)=E_1$. As $S^* \in \syl_2(C_G(R_1))$,  we may apply Lemma~\ref{AB} to the action of
$C_{M}(R_1)$ on $E_1/R_1$ to see that $O^{2'}(P_1)$ is a characteristic subgroup of $C_{M}(R_1)$. Thus $N_0$
normalizes $O^{2'}(P_1)$ and similarly $N_0$ normalizes $O^{2'}(P_2)$ and, as $H^*= \langle O^{2'}(P_1),
O^{2'}(P_2)\rangle$,  this proves our claim. Finally we now have that $N_0$ normalizes $H$ as does $Q$ and thus
$N_G(Q) \le H$ and we have our contradiction.
\end{proof}

\begin{lemma}\label{q=2} If $H^*\cong \L_3(2^a)$, then $a> 1$.
\end{lemma}

\begin{proof} Suppose $n = 1$. Then either $Q$ is elementary abelian of order 4, a dihedral group of order 8,
a {dihedral} group of order 16, or a cyclic group. In all of these cases we have $N_G(Q) \leq H$, which is a
contradiction.
\end{proof}

\begin{lemma}\label{a3} Assume $H^* \cong  \L_3(2^{a})$  and that one of the following holds:
\begin{itemize}
\item[(i)] $a > 2$; or
\item[(ii)] $O_2(N_G(Z(S^*))) \leq S^*$.
\end{itemize}
 Then $ F^\ast(C_G(Z(S^*))) = S^*$.
\end{lemma}

\begin{proof}  As $Z(S^*)$ is normal in $S$  and $G$ has parabolic characteristic $2$, $F^*(N_G(Z(S^*)))
= O_2(N_G(Z(S^*)))$.

We have that $$[O_2(N_G(Z(S^*))), N_{H^*}(S^*)] \le O_2(N_G(Z(S^*)))\cap H^* \le S^*.$$ If $a > 2$, $\Out(H^*)$
acts faithfully on $N_{H^*}(S^*)/S^*$ and so we conclude that $O_2(N_G(Z(S^*))) \le S^*$. So we may assume that
we have (ii) holds. But then  $$[O_2(N_G(Z(S^*))),S^*]\le Z(S^*)$$ and so $S^*= O_2(N_G(S^*)))$ as claimed.
\end{proof}

\begin{lemma}\label{NGE} Assume $H^* \cong \L_3(2^a)$. Let $S^*= E_1E_2$ where $E_1$ and $E_2$ are the elementary
abelian subgroups of order $2^{2a}$ in $S^*$. Then either  $N_G(E_1) \not\leq H$ or $N_G(E_2)\not \le H$.
\end{lemma}

\begin{proof}  Assume that  $N_G(E_1)$ and $N_G(E_2)$ are both contained in $H$. {Then $Q \not= E_1$ and $Q\not = E_2$.}
 By Lemma \ref{q=2}  $a> 1$. Set  $U = Z(S^*)=E_1 \cap E_2$. We remark that every involution of $S^*$ is
 contained in either $E_1$ or $E_2$ and that, for $i=1,2$, $O^{2'}(N_{H^*}(E_i)/E_i)\cong \SL_2(2^a)$ with $E_i$ a natural $\SL_2(2^a)$-module.

 First we show that

 \begin{claim}\label{UinQ} $U \leq Q.$\end{claim}

 If $Q \le S^*$, then we
simply have $U= Z(S^*) \le C_G(Q) \le Q$. If $Q \not \le S^*$, then we argue that there is an element of $x \in
[Q,S^*] \setminus E_1$ and note that for such elements we have $U=[x,E_1]\le Q$. Thus in both cases $U \le Q$.
\qedc

\begin{claim} \label{clm}$N_G(Q) \leq N_G(U).$\end{claim}
\medskip

Aiming for a contradiction suppose that $N_G(U) \not\ge N_G(Q)$.

Assume  first that $Q \not\leq H^*$. Suppose  additionally that  $a>2$. If there is an element $x$ of $Q$ which
induces a field automorphism of order $2$ on $H^*$, then, for $i=1,2$, $[E_i,x]U\le Q$ has order $2^{3a/2}$.
Hence $|E_1 \cap Q| \geq 2^{3a/2} \leq |E_2 \cap Q|$.    Let $F$ be an elementary abelian subgroup of $Q$
with $|F| \geq 2^{3a/2}$.  Assume that $F \not \le S^*$ and let $y \in F \setminus S^*$. Then $|FS^*/S^*| \le 4$
and $C_{S^*}(y)\ge F\cap S^*$ indicates that $|C_{S^*}(y)|\ge 2^{3a/2-2}$. On the other hand, the $2$-rank of
$C_{S^*}(y)$ is at most $a$. Hence, as $a$ is even,  $a = 4$, $|FS^*/S^*|=4$ and $|F| = 16$. But then $C_{E_1}(x)
\leq F$ or $C_{E_2}(x) \le F$. As $F$ contains a graph automorphism this is impossible. Thus every elementary
abelian subgroup of $Q$ of order at least $2^{3a/2}$ is contained in $S^*$. Hence $E_1 \cap Q$ and $E_2 \cap Q$
are the unique elementary abelian subgroups of $Q$ of their order and so $U=(E_1\cap Q)\cap (E_2\cap Q)$ is
normalized by $N_G(Q)$ which is a contradiction. Therefore we may suppose that $Q$ contains no  elements which
act as field automorphisms on $H^*$. In particular, we have $|QS^*/S^*|=2$. If $Q$ contains a graph automorphism,
then $Q$ centralizes $U$ and therefore $N_{H^*}(U)$ normalizes $Q$. But then we have $[N_{H}^*(Q),Q]\le Q \cap
H^* \le S^*$ contrary to the fact that the graph automorphism does not centralize a torus of $H^*$ when $a> 2$.
Thus the elements of $Q \setminus S^*$ induce graph-field automorphisms on $H^*$. Now  $U$ is a
normal subgroup of $Q$ and that $Q$ has no other normal elementary abelian subgroup of the same order. Thus $U$
is normal in $N_G(Q)$ in this case as well contrary to our assumption. Consequently $a=2$.

Now suppose that $a= 2$. Then  we have  $|Q/S^*| \le 4$ and by \ref{UinQ} $Q \ge U$ with $U$ elementary abelian
of order $4$. Let $W = \langle U^{N_G(Q)}\rangle$ and assume that $W\neq U$. Suppose that $W \le S^*$. Then $W$
centralizes $U$ and is therefore abelian and hence  elementary abelian. So without loss of generality we may
assume that $W\le E_1$. Then $H$ is contained in $H^*$ extended by a field automorphism.    If $W= E_1$, then
$N_G(Q) \le N_G(E_1)\le H$, which is impossible. Hence $W$ has order $2^3$. Note that $W$ is normal in $S$ and so
$N_G(W)$ is of characteristic $2$.  As $C_{S}(W) = E_1$, we get that $N_G(W) = N_{N_G(W)}(E_1)C_G(W)$. As $|E_1 :
W| = 2$, we have that $C_G(W)$ has a normal 2-complement $J$. But then $J$ is also normal in $N_G(W)$, which is
of characteristic 2. So $J = 1$ and  $C_G(W) = E_1$. This now shows that $N_G(Q) \le N_G(W) \leq
N_G(E_1) \leq H$, a contradiction.  Thus we have $W \not \le S^*$. Then $W \cap S^*$ contains elements of order
$4$. Since $W$ is generated by involutions it follows that $W$ is not abelian and $U$ is not in the centre of
$W$. Let $u\in W$ be conjugate to an element of $U$ and assume that $u \not \in H^*$ and   $u$ does not
centralize $U$. Then $C_{H^*}(u) \cong \U_3(2)$ or $\PSL_3(2)$. As $u$ is conjugate to an element of $Z(Q)$,  $W$ is normalized by such a group. Since $\U_3(2)$ has no $\GF(2)$-representations of dimension less
than $8$ and $|S|\le 2^8$ is non-abelian, we must have that $u$ acts as the field automorphism of $H^*$. Because
$|N_G(Q)/Q| \ge |\PSL_3(2)|_2=2^3$,  we have that $|Q|\le 2^5$. This means that $Q=W$ has order $2^5$. Let $ L
\cong \PSL_3(2)$ act on $Q$. Then $[L,Z(Q)] = 1$. If $|\Omega_1(Z(Q))| = 4$, then $L$ acts transitively on
$Q/\Omega_1(Z(Q))$ and so as $Q$ is generated by involutions $Q$ is elementary abelian, a contradiction. So
$Z(Q)$ is cyclic. If $X$ is normal in $Q$ and $L$-invariant and $|X| = 4$, then $L$ acts transitively on $Q/X$
and so $X \leq Z(Q)$. Hence in any case there is $Y \leq Q$, $Z(Q) \leq Y$, $|Y/Z(Q)| = 8$ and $L$ acts
transitively on $Y/Z(Q)$. As $L$ acts 2-transitively on $Y/Z(Q)$, we then get that $Y$ is abelian. As $Q$ is
non-abelian,  $Y \not= Q$ and so $Y$ is elementary abelian. But now $[Q,Y]$ is trivial or isomorphic
to $Y/Z(Q)$. The latter is not possible and so $[Y,Q] = 1$, which gives the contradiction $Q$ abelian. Therefore
\ref{clm} holds when $Q \not \le H^*$.

Assume now that $Q \leq H^*$. Then $[U,Q] =1$ and so $U \leq Z(Q)$. If $\Omega_1(Z(Q)) =U$ we are done, so
$\Omega_1(Z(Q))\not= U$. Since $Q$ is normalized by $N_G(S^*)$, we now may suppose that $Q = E_1$. But then by
assumption $N_G(Q) \leq H$, which is a contradiction. This proves \ref{clm}. \qedc
\\
\\
We now consider $N_G(U)$. As $U$ is normal in $S$ {and $G$ has parabolic characteristic $2$,} $C_G(O_2(N_G(U)))
\leq O_2(N_G(U))$.

Suppose that $O_2(N_G(U)) \leq  S^*$.  Then, by Lemma \ref{a3}, $S^*= O_2(N_G(U))$ and we conclude that $N_G(U)$
has a subgroup $N$ of index $2$ which normalizes $E_1$. But then $N \le H$ by hypothesis and so also $$N_G(Q) \le
N_G(U) = NS \le H$$ by \ref{clm}, a contradiction.

 Thus $O_2(N_G(U)) \not \le S^*$. Then, by Lemma \ref{a3}, we have $a = 2$. Let $T$ be a complement to $S^*$ in $N_{H^*}(S^*)$. Then $T$ normalizes $U$ and
hence it also normalizes $O_2(N_G(U))$. But then $$[T,O_2(N_G(U))] \le O_2(N_G(U)) \cap H^* \le S^*$$  and
$O_2(N_G(U))$ contains an element which acts as a graph automorphism on $H^*$.  Now just as in the last paragraph
$E_1 \not\leq O_2(N_G(U))$ and so $U = \Omega_1(N_G(U) \cap S^*)$. As $[E_1,O_2(N_G(U))]E_1= S^*$, we have
$O_2(N_G(U)) \cap S^* \cong 4 \times 4$  is characteristic in $O_2(N_G(U))$. But then $|N_G(U)/C_G(U)|_{2'}= 3$
  and so $N_G(U) = C_G(U)TS = C_G(U)N_H(U)$. As $C_G(U)$ is a 2-group we get $N_G(U) = TS \leq H$.  Using
\ref{clm} we get $N_G(Q)\le N_G(U) \leq H$.
\end{proof}
%


\begin{lemma}\label{n=2} If  $H^* \cong \L_3(2^a)$, then $a = 2$. \end{lemma}

\begin{proof} By Lemma \ref{q=2} we have $a > 1$. So assume that $a>2$.
By Lemma \ref{NGE} there is an elementary abelian subgroup $E_1\le S^*$ of order $2^{2a}$ such that $N_G(E_1)
\not\leq H$ and $O^{2'}(N_H(E_1)/E_1) \cong \SL_2(2^a)$.  Let $i\in N_S(E_1)$ be an involution which projects
non-trivially in $O^{2'}(N_H(E_1)/E_1)$. Then $C_{S \cap N_H(E_1)}(i)$ contains an elementary abelian subgroup
$E_2$ and $C_{S \cap N_H(E_1)}(i)/E_2E_1$ is cyclic. Suppose that  $j\in N_G(E_1)$ is an involution which
projects into $N_S(E_1)/E_1$,
 but not into $O^{2'}(N_H(E_1)/E_1)$. Then $j$ acts as a field automorphism on $\L_3(2^a)$ and so
 $C_{N_S(E_1)}(j)$ is an extension of Sylow 2-subgroup of $\L_3(2^{\frac{a}{2}})$ by a
cyclic group. Since in this case $a \geq 4$, this latter group does not contain an elementary
abelian group with cyclic factor group. Hence $i$ and  $j$ are not conjugate in $N_G(E_1)$. Using  Thompson's
Transfer Lemma \cite[Theorem 15.16]{GLS2}, there is a normal subgroup $N$ of $N_G(E_1)$ which has Sylow
$2$-subgroup contained in $N_{H^*}(E_1)/E_1$. Thus, employing Lemma~\ref{AB},
$O^{2'}(N_{H^*}(E_1))C_G(E_1)$ is a normal subgroup of $N_G(E_1)$.

Suppose that $1 \not= \omega \in C_G(E_1)$ has odd order. Then $[Z(S^*), \omega] = 1$.  As $a \ge 3$, we get with
Lemma \ref{a3} that $O_2(C_G(Z(S^*))) = S^*$.  Thus $\omega$ normalizes $S^*$ and so Lemma~\ref{w=1} and
$F^*(N_G(Z(S^*)))=O_2(N_G(Z(S^*)))$ together imply $\omega=1$. Therefore $O^{2'}(N_{H^*}(E_1))$ is a normal
subgroup of $N_G(E_1)$ and by the Frattini Argument
$$N_G(E_1)= N_G(S^*) O^{2'}(N_{H^*}(E_1)).$$ Similarly we have $$N_G(E_2)=N_G(S^*) O^{2'}(N_{H^*}(E_2)).$$ In
particular, we now have $N_G(S^*)$ normalizes $$\langle O^{2'}(N_{H^*}(E_1)), O^{2'}(N_{H^*}(E_2))\rangle =
H^*.$$ But then $N_G(E_1)$ and $N_G(E_2)$ are contained in $H$ and this contradicts Lemma~\ref{NGE}.
 %
%
%
\end{proof}

\begin{proposition}\label{M23} If $H^* \cong \L_3(2^a)$, then  $a=2$ and $G \cong \Mat(23)$.
\end{proposition}

\begin{proof} By Lemma \ref{n=2} we have $H^* \cong \L_3(4)$ and because of
 Lemma \ref{NGE} we may assume  that $N_G(E_1) \not\leq H$.
 Suppose that $i\in \{1,2\}$ and $1 \not= \omega \in C_G(E_i)$ is of odd order.
Then $[U, \omega] = 1$. As $Z(Q) \cap U \neq 1$, we have that $C_G(E_i) \le N_G(Q)$. In particular $\omega$
normalizes $Q \cap N_G(E_i)$ and so centralizes $Q \cap N_G(E_i)$. As $|Q : Q \cap N_G(E_i)| \leq 2$, we get $[Q,
\omega] = 1$, a contradiction. Therefore  $C_G(E_i) = E_i$ for both $i=1,2$. This yields that $N_G(E_i)/E_i$ is
isomorphic to a subgroup of $\SL_4(2)$ which strictly contains $\Alt(5)$ or $\Sym(5)$ as a subgroup of odd index.
Inspection of the subgroups of $\SL_4(2)$  shows that $$N_G(E_1)/E_1 \cong \Alt(5) \times 3, (\Alt(5) \times 3):2
\;\; \text { or } \Alt(7).$$ If $N_{H^*}(E_i)$ is normal in $N_G(E_i)$ for both $i = 1,2$, then $N_G(E_i)=
N_{N_G(E_i)}(S^*)N_{H^*}(E_i)$. This then gives $N_{N_G(E_1)}(S^*)=N_{N_G(E_2)}(S^*)$ which then means  that
$N_{N_G(E_1)}(S^*)$ normalizes $\langle N_{H^*}(E_1), N_{H^*}(E_2) \rangle = H^*$  and so  $N_G(E_i) \le N_G(H^*)
= H$ which is a contradiction. Hence we may assume that $$N_G(E_1)/E_1 \cong\Alt(7).$$
 Now taking $z \in Z(S)$, we get that $C_{N_G(E_1)}(z) = E_1L$ where $L \cong \SL_3(2)$ and so $Q =
E_1$ and $C_G(z)= E_1L$. Hence $C_G(z) \leq N_G(E_1)$.  It also follows that $N_G(E_2)/E_2 \not \cong \Alt(7)$
(for otherwise  $E_1 = Q=E_2$) and so $N_G(E_2)$ normalizes $N_{H^*}(E_2)$. Since $N_{N_G(E_1)}(E_2) =
N_{N_G(E_1)}(E_1 \cap E_2)$ and $\Alt(7)$ acts transitively on the subgroups of $E_1$ of order $2^2$, we have
$$N_{N_G(E_1)}(E_2) /E_1 \approx (3 \times \Alt(4)):2$$ and $N_{N_G(E_1)}(E_2) = N_G(E_1) \cap N_G(E_2)$. It
follows that
$$N_G(E_2)/E_2 \approx (3\times \Alt(5)):2.$$ Since $N_G(E_1) = \langle N_{N_G(E_1)}(E_2), N_{H^*}(E_1)\rangle$,
we have that $H\approx \PSL_3(4){:}2$.

Let $B \le N_G(E_1)$ be such that $B/E_1 \cong \Alt(6)$ and $B\cap H^* \approx 2^4{:}\SL_2(4)$. Let $U = \langle
H^*, B\rangle$. We have $N_B(E_2) = N_B(E_1 \cap E_2) \approx 2^4{:}\Sym(4)$  and $N_B(E_2) \not \le H$. We set
$C =  N_B(E_2)N_{H^*}(E_2)\le N_G(E_2)$. Then $C/E_2 \cong \Sym(5)$.   Applying Lemma~\ref{M22} gives $U \cong
\Mat(22)$.  Now we consider the triangle of groups consisting of  $U$, $N_G(E_1)$ and $N_G(E_2)$ and apply
Lemma~\ref{M23r}  with $B = N_G(E_1)$ and $P = U$ to get  $M=\langle P, U \rangle \cong \Mat(23)$. In particular,
we now know that $G$ has one conjugacy class of involutions, the fusion of these involutions is controlled in $M$
and $C_G(x) \le M$ for all involutions $x$ in $M$.   Thus, if $M<G$, then $M$ is strongly $2$-embedded in $G$
(\cite[Proposition 17.11]{GLS2}).  Since, by  \cite{Be}, $G$ does not have a strongly $2$-embedded subgroup  we
infer that $G=M$ and this completes the proof of the proposition.
\end{proof}

We conclude this section by proving Theorem~\ref{main} when $p=2$.

\begin{proof}[Proof of Theorem~\ref{main} when $p=2$.] Lemma~\ref{possible}(i) and (ii) provides a list of
candidates for $H^*$. The configurations with   $H^*\cong  \Omega_6^+(2)$, $\U_4(2)$, $\U_5(2)$, $ \Sp_6(2)$,
${}^2\F_4(2)$, $\Sp_4(2^a)'$ for some $a\ge 2$ and $\PSL_3(2^a)$ for $a\neq 2$
 are shown to be impossible in Lemmas~\ref{O6},  \ref{U42}, \ref{U52}, \ref{sp62}, \ref{twistedF4}, \ref{Sp4} and
 \ref{n=2}
respectively. The remaining possibilities are that $H^* \cong \Omega_8^+(2)$, $\G_2(2)'$, $\Sp_4(2)'$ or
$\PSL_3(4)$. In these cases Propositions~\ref{O8}, \ref{G2}, \ref{q=2sp4} and \ref{M23}   show that $F^*(G)$ is
as described in Theorem~\ref{main} (ii).
\end{proof}

\section{The configurations with  $p = 3$}

In this section we assume that $G$ and $H$ are as in the statement of the main theorem and that in addition
$p=3$. We adopt the notation from Section~3 and investigate each of the groups listed in Lemma~\ref{possible} with the exception of $\PSL_3(3^a)$.

 \begin{lemma}\label{sp43}  We have $H^*\not \cong \PSp_4(3)$.
\end{lemma}

\begin{proof} Suppose that $H^* \cong \PSp_4(3)$. Then $S=S^* \le H^*$ and so,
letting $Z=Z(S)$, we have $Q\le O_3(N_H(Z))$. Since $O_3(N_H(Z))\approx  3^{1+2}_+$ and $N_{H^*}(Z)/O_3(N_H(Z))
\cong \SL_2(3)$, we  have $Q= O_3(N_H(Z))$.  Therefore  $\Out(Q) \cong \GL_2(3)$  and so $N_H(Z)$ is normalized
by $N_G(Q)$. Then $N_G(Q)=N_{N_G(Q)}(S)N_H(Q)$ and $N_{N_G(Q)}(S)$ normalizes the unique abelian
subgroup $E$ of $S$ of order $3^3$.  From the structure of $\PSp_4(3)$, we get $N_{H^*}(E)/E \cong \Alt(4)$ and
$C_G(E)=E$ as $E$ is normal in $S$ and $G$ has parabolic characteristic $3$. Thus $\langle N_{H^*}(E),
N_{N_G(Q)}(S)\rangle$ embeds into $\GL_3(3)$ and has Sylow $3$-subgroups of order $3$ and non-trivial Sylow
$2$-subgroups. Surveying  the maximal subgroups of $\GL_3(3)$ \cite{Atlas} shows that $N_{N_G(Q)}(S)$ normalizes
$N_{H^*}(E)$. But then $N_{N_G(Q)}(S)$ normalizes $\langle N_{H^*}(E), N_{H^*}(Z)\rangle = H^*$. Therefore
$N_G(Q) \le H$, which is a contradiction.
\end{proof}

{\begin{lemma}\label{notG2} We have $H^* \not \cong \G_2(3^a)$.
 \end{lemma}

\begin{proof}  Suppose that $H^* \cong \G_2(3^a)$. Then $Z(S) \cap S^*$ contains both long and short root elements $z_1$ and $z_2$  in $Z(S^*)$. Thus $N_G(Q) $ contains $\langle C_{H^*}(z_1), C_{H^*}(z_2)\rangle= H^*$ which is impossible.
 \end{proof}}


\begin{proposition}\label{L43} Suppose that  $H^*\cong \L_4(3)$ or $\U_4(3)$. Then $H^* \cong \U_6(2)$, $\F_4(2)$, $\McL$ or $\Co_2$.
\end{proposition}

\begin{proof} Suppose that   $H^*\cong \L_4(3)$ or $\U_4(3)$. Then $S= S^*$ and we have $z \in Q \cap Z(S)$
with  $O_3(C_H(z)) \cong 3^{1+4}_+$. Furthermore  $E= J(S)$ is elementary abelian of order $3^4$ and   (see
\cite[pages 69 and   52]{Atlas})
$$N_{H^*}(E)/E \cong \begin{cases}(\SL_2(3) \ast \SL_2(3)):2& \text { if } H^* \cong \L_4(3)\\\PSL(2,9) \cong \Alt(6)& \text{  if
} H^* \cong \U_4(3)\end{cases}.$$ In both cases
 an inspection of the maximal subgroups of $\GL_4(3)$ \cite{Atlas} yields

\begin{claim}\label{cl2}  $ O^{3^\prime}(N_H(E))$ is normal in $N_G(E)$ and $N_G(E) = N_G(S)O^{3'}(N_H(E))$. \end{claim}
\medskip

We next determine  $Q$. Suppose that $Q \not= O_3(C_H(z))$. Then $Q$ is a normal subgroup of $C_H(z)$, which is
properly contained in $O_3(C_H(z))$ and different from $\langle z \rangle$. Then Lemma~\ref{SU43} gives
$H^* \cong \PSL_4(3)$ and $|Q| = 3^3$. Furthermore, by Lemma~\ref{SL43}, there are exactly $4$ possibilities for
$Q$, two of them are extraspecial and two are elementary abelian.  If $Q \cong 3^{1+2}_+$, then
$C_{O_3(C_H(z))}(Q)Q= O_3(C_H(z))
> Q$, which is a contradiction.  Hence $Q$ is elementary abelian and $C_G(Q) =Q$. Furthermore, $N_H(Q)/Q
\cong \SL_3(3)$ and so, as  $N_G(Q)>N_H(Q)$ we must have $N_G(Q)/N_H(Q)=2$ and $Z(N_G(Q)/Q)= 2$. Let $Y$ be the
preimage of $Z(N_G(Q)/Q)$. Then $Y$ normalizes $O_3(C_H(z))$ and hence $Y$ normalizes $N_H(E_2)$ where $E_2$ is
an elementary abelian normal subgroup of $N_G(\langle z \rangle)$ contained in $Q$ with $E_2 \neq Q$. But then
$Y$ normalizes $H^*=\langle N_{H^*}(Q), N_{H^*}(E_2)\rangle $ and consequently $N_G(Q) \le H$, which is a
contradiction.

%
%

Therefore $Q = O_3(C_H(z))$ is extraspecial of order $3^5$ and exponent $3$.  This yields that  $N_G(Q)/Q$ is
isomorphic to a subgroup of $\GSp_4(3)$, which has a Sylow 3-subgroup of order 3.  Employing either \cite{Atlas}
or \cite{BNpairs} we get that one of the following holds:
\begin{enumerate}
\item[(1)]  $N_G(Q)/Q \cong 2^{1+4}_-.\Alt(5)$ or $2^{1+4}_-.\Sym(5)$;
\item[(2)] $E(N_G(Q)/Q) \cong \SL_2(5)$; or
\item[(3)] $|N_G(Q)/Q| = 2^a\cdot 3$.
\end{enumerate}
 If (1) occurs, then, as $Z(S)$ is not weakly closed in $S$ with respect to $G$,  \cite{Co2}  yields  $G \cong \Co_2$.\\
 \\
Suppose we have possibility (2). Assume further that $H^* \cong \PSL_4(3)$. We will show $N_G(S) \leq H$.

 We know that $N_G(S)$ normalizes $E$ and by \ref{cl2} also $O^{3^\prime}(N_H(E))$.
In $S$ there are two elementary abelian subgroups $E_1$, $E_2$ of order $3^3$ such that, for $i=1,2$,  $U_i =
\langle Q^g~|~Z(Q)^g \leq E_i \rangle$ satisfies $$U_i/E_i \cong \SL_3(3).$$  In fact, $E_1$ is the group of
transvections to a point and $E_2$ the group of transvections to a hyperplane containing this point.  Hence
$E_iE/E$ correspond to the two subgroups of order three in $S/E$, which act quadratically on $E$. In particular
$N_G(S)$ acts on $\{E_1E, E_2E \}$. We have that  $O^{3^\prime}(N_H(E))$ contains an involution $x$ which inverts
$E$. Let $M= N_{N_G(S)}(E_1E)$.  We factor $M= C_M(x)E$. Then, for $i=1, 2$,  $M$ normalizes $Z(E_iE) = E_i\cap
E$ which has order $3^2$. Now $E_i=C_{E_iE}(x)(E_i\cap E)$ is normalized by $C_M(x)$. Since $E$ normalizes $E_i$,
we infer that $E_i$ is normalized by $M$. Therefore  $N_G(S)$ permutes $\{E_1, E_2\}$ and normalizes $\langle
U_1, U_2 \rangle = H^*$. Hence by assumption we then have that $N_G(S) \leq H=N_G(H^*)$.

 Now generally if (2) holds, then  $$N_G(Q)= \langle N_H(Q), N_{N_G(Q)}(S) \rangle,$$  as
 $ N_H(Q)/Q \cap E(N_G(Q)/Q) \cong \SL_2(3)$ and $N_{E(N_G(Q)/Q)}(S/Q) \approx 3:4$ and
 together these groups generate $E(N_G(Q)/Q)$.  Hence as $N_G(Q) \not\leq H$, we get that $F^\ast(H) \cong \U_4(3)$ and so $E(N_G(E)/E) \cong \Alt(6)$.
Finally, using  Lemma \ref{McL} yields $F^\ast(G) \cong \McL$.

So we may assume that we have possibility  (3).   The Frattini Argument delivers
$$N_G(O^{3^\prime}(N_H(Q))) = N_{N_G(O^{3^\prime}(N_H(Q)))}(S)O^{3^\prime}(N_H(Q)).$$ The left factor normalizes
$O^{3^\prime}(N_H(E))$   by \ref{cl2}. Therefore $N_G(O^{3^\prime}(N_H(Q)))$  normalizes $\langle
O^{3^\prime}(N_H(E)), O^{3^\prime}(N_H(Q)) \rangle = H^*$ and so is contained in $H$. Thus
$N_{N_G(Q)}(O^{3^\prime}(N_H(Q))) \le H$. Using this information and when inspecting the subgroups of $\GSp_4(3)$
given in \cite{Atlas} we obtain
$$O^{3^\prime}(N_G(Q)/Q)N_H(S) \leq R \cong (\Q_8 \times \Q_8). \Sym(3).$$
Hence, as $N_G(Q) \not\leq H$, we get that $R$ is isomorphic to a subgroup of $N_G(Q)/Q$. In particular  $N_G(Q)/Q$ is a subgroup of the subgroup of $\GSp_4(3)$ which preserves a
decomposition of the natural $4$-dimensional symplectic space over $\GF(3)$ into a perpendicular sum of two
non-degenerate $2$-spaces.  We further see that $O^{3^\prime}(N_G(Q)/Q)$ is isomorphic to a subgroup of $\Sp_2(3)
\times \Sp_2(3)$ which  projects nontrivially on to both direct factors. In particular  $O^{3^\prime}(N_G(Q)/Q)$  contains a normal subgroup isomorphic to $\Q_8 \times \Q_8$.

Therefore, if either  $H^* \cong \U_4(3)$ or  $H^* \cong \L_4(3)$, then  $Z(Q)$ is not weakly closed in $S$. Hence in case (3) we have that $F^*(G) \cong \U_6(2)$ or $\F_4(2)$ by  Lemma \ref{F42}.
\end{proof}

\begin{proposition}\label{O73} Suppose that $H^*\cong \Omega_7(3)$. Then $G^* \cong {}^2\E_6(2)$ or $\M(22)$.
\end{proposition}

\begin{proof} Again $S=S^*$. We set  $Z= Z(S)$ and note that $$N_{H^*}(Z) \approx 3^{1+6}_+.(\SL_2(3) \times \Omega_3(3)).2.$$  As a module for this group $O_3(C_H(Z))/Z$ is
the tensor product of the natural $\SL_2(3)$-module with the 3-dimensional orthogonal $\Omega_3(3)$-module and
this is an irreducible action. Therefore $Q = O_3(N_{H^*}(Z))$. Inspection of the irreducible subgroups of
$\Sp_6(3)$ (see \cite{Atlas}) shows that $N_G(Q)/Q$ is a subgroup of $U = (\Sp_2(3) \wr \Sym(3)):2$. As obviously
$\Omega_1(O_2(U)) \leq N_{H^*}(Q)/Q$ we either see that the assumptions of Lemma~\ref{2E6}  or Lemma~\ref{M221} are satisfied and so we
have the assertion or $|N_G(Q) : N_{H^*}(Q)| = 2$.  So assume that  $|N_G(Q) : N_{H^*}(Q)|=2$. Then $N_G(Q)/Q
\cong \GL_2(3) \times \Sym(4)$ and $N_G(Q)=N_G(S)N_{H^*}(Q)$. Let $P\le N_G(Q)$ be such that $P/O_3(P) \cong
\GL_2(3)\times 2$ and note that $O_3(P) = QO_3(L)$ where $L$ is the parabolic subgroup of $H^*$ which contains
$S$ and has shape $3^{3+3}:\SL_3(3)$. We have that the preimage of $C_{Q/Z(Q)}(O_3(P))$ is equal to $E=Z(O_3(L))$
is elementary abelian of order $3^3$ and is normalized by $N_G(S)$. As $E$ is normal in $S$,
$O_3(N_G(E)) = O_3(N_{H^*}(E))$. This yields that $N_{H^*}(E)$ is normal in $N_G(E)$, as  $N_{H^*}(E) =
O^{3^\prime}(N_G(E))$. Then $N_G(S)$ normalizes $\langle N_{H^*}(E), N_{H^*}(Z(Q)) \rangle = H^*$. But then by
assumption $N_G(Q)= N_G(S)N_{H^*}(Q) \le H$, which is a contradiction. This proves the proposition.
\end{proof}

We finally consider the configurations with  $H^* \cong \mathrm P\Omega^+_8(3)$ and  do this though a series of lemmas. Set $Z=
Z(S^*)=Z(S)$.  We have that $Z$ has order $3$ and
$$N_{H^*}(Z)/O_3(N_{H^*}(Z)) \approx  (\SL_2(3)\ast \SL_2(3) \ast \SL_2(3)):2 \approx 2^{1+6}_-.3^3.2,$$  as can be
seen in \cite{Atlas}.  We also recall that $H/H^*$ embeds into $\Out(H) \cong \Sym(4)$. The action of
$N_{H^*}(Z)$ on $O_3(N_{H^*}(Z))$ is as a tensor product of the natural $\SL_2(3)$-module with the
four-dimensional orthogonal module for $\mathrm O_4^+(3)$. In particular, $N_{H^*}(Z)$ acts irreducibly on
$O_3(N_{H^*}(Z))/Z$ which has order $3^8$.

Hence we get

\begin{lemma}\label{O831} Suppose that $H^*\cong \mathrm P\Omega^+_8(3)$. Then
\begin{enumerate}
\item $Q=O_3(N_{H^*}(Z)) $ is extraspecial of order $3^9$ of exponent $3$ and $N_G(Q)/Q$ is isomorphic to a subgroup of $\GSp_8(3)$.
{\item $S^*/Q $ is elementary abelian of order $3^3$ and $|S/S^*|\le 3$.
\item $N_{H^*}(S^*)/S^*$ has order $2$.
\item $N_H(Q)/Q$ is isomorphic to a subgroup of a group of shape $(\GL_2(3) *\GL_2(3)*\GL_2(3)). \Sym(3)$ embedded in $\GSp_8(3)$. In particular, $N_H(Q)$ is a $\{2,3\}$-group. }\end{enumerate}\qed
\end{lemma}

We remark that the subgroup of $\GSp_8(3)$  in Lemma~\ref{O831} (iv) is precisely (and uniquely) described in
\cite[Section 3]{PaS}.

\begin{lemma}\label{O832} Suppose that  $H^*\cong \mathrm P\Omega^+_8(3)$ and $E \leq S^*$ is an elementary abelian
subgroup  of $S^*$ of order $3^6$ such that $N_{H^*}(E)/E\cong \Omega^+_6(3)$. Then $N_{H^*}(E)$ is normal in
$N_G(E)$ and $N_G(E)=N_{N_G(E)}(S^*)N_{H^*}(E)$.
\end{lemma}

\begin{proof} First  suppose that $\omega $ is a $3^\prime$-element which centralizes $E$. Then
$[\omega, Q \cap E] = [\omega,Z]=1$ and so $\omega$ normalizes $Q$ and centralizes a maximal abelian subgroup in
$Q$. Thus $[Q,\omega] = 1$ by Lemma~\ref{w=1} and consequently $\omega = 1$.

Let $e \in E$ correspond to a non-singular point  in $E$ and assume that  $e$ is conjugate to $z$ in $N_G(E)$.
Then $C_{N_{H^*}(E)}(e)/E$  has a normal subgroup isomorphic to $\Omega_5(3)\cong \PSp_4(3)$. As $|S : Q| \leq
3^4$, we see that $E \leq O_3(C_G(e))$. But $Q$ does not contain an elementary abelian group of order $3^6$. So
 $N_{H^*}(E)$ controls fusion of the $N_G(E)$-conjugates of $z$ in $E$ and this yields $N_{H^*}(E) =
\langle Q^{N_G(E)} \rangle$. In particular $N_{H^*}(E)$ is normal in $N_G(E)$ and $N_G(E) =
N_{H^*}(E)N_{N_G(E)}(S^*)$ as claimed.
\end{proof}

\begin{lemma}\label{O833} Suppose that  $H^* \cong \mathrm P\Omega^+_8(3)$. Then $N_G(S^*) = N_H(S^*)$.
\end{lemma}

\begin{proof}  We have that $Q$ is
normalized by $N_G(S^*)$ and $S^*/Q$ is elementary abelian of order $3^3$. In $S^*$ there are three elementary
abelian subgroups of order $3^6$,  $E_1$, $E_2$, $E_3$, whose normalizer in $H^*$ involves $\Omega^+_6(3)$.
Furthermore the groups $QE_i$, $i = 1,2,3$, correspond to three different subgroups of order three in  $S^*/Q$.
By \cite[Lemma 3.1(i)]{PaS} these are the only three subgroups of order three in $S^*/Q$, which centralize an
elementary abelian subgroup of order $3^5$ in $Q$. Therefore $N_G(S^*)$ permutes $E_1$, $E_2$ and $E_3$. Hence,
by Lemma \ref{O832}, $N_G(S^*)$ normalizes $\langle N_{H^*}(E_i)~|~ i = 1,2,3 \rangle = H^*$. By our general
assumption we then have $N_G(S^*) \leq H$.
\end{proof}

\begin{lemma}\label{O834} Suppose that $H^* \cong \mathrm P\Omega^+_8(3)$. Then $N_G(Z_2(S)) \leq H$.
\end{lemma}

\begin{proof} From \cite[Lemma 3.1 (v)]{PaS} $Z_2(S)$ has order $9$.
Assume that $g\in H^*$ and $z^g\in Z_2(S)$. Then $z^g = z^h$ for some $h \in H^*$ and therefore
 $$P = \langle Q^g~|~g \in G, z^g \in Z_2(S) \rangle \leq H^*$$ which means that $P$ is normal in
 $N_G(Z_2(S))$. Hence, from the structure of $\mathrm P\Omega_8^+(3)$
 we have $P\ge S^*$ and
 $ N_G(Z_2(S))= PN_G(S^*)$. Finally Lemma \ref{O833} yields then $N_G(Z_2(S)) \le H$.
\end{proof}

Set $$X= O_{3,2}(N_H(Q)).$$ Our  objective over the next few lemmas is to show that $N_G(Q)= N_G(X)$.

\begin{lemma}\label{O835} Suppose that $H^* \cong \mathrm P\Omega^+_8(3)$. Then $X/Q$ is
extraspecial of order $2^7$ and of $-$-type and  one of the following holds: \begin{enumerate}
\item$N_{C_G(Z(Q))}(X)/X \cong \U_4(2)$ or $3^{1+2}_+.\SL_2(3)$; or \item $N_{C_G(Z(Q))}(X) \leq H$.\end{enumerate}
\end{lemma}

\begin{proof} We have already commented that $X/Q$ is a central product of three subgroups isomorphic to $\Q_8$.
By Lemma \ref{extra},  $N_{C_G(Z(Q))}(X)/X$ is isomorphic to a subgroup of $\U_4(2)$ which has order
 divisible by $3^3$. If $N_{C_G(Z(Q))}X/X$ normalizes $S^*X/X$, then $$N_{C_G(Z(Q))}(X)\le N_G(S^*)X \le H$$ by
Lemma~\ref{O834}. This is (ii). Employing  the subgroup structure of $\U_4(2)$  as given in \cite{Atlas} now
delivers the assertion.
\end{proof}

\begin{lemma}\label{O837+} Suppose that  $H^* \cong \mathrm P\Omega^+_8(3)$. If $N_{C_G(Z(Q))}(X)/X \cong \U_4(2)$, then
$N_G(X) = N_G(Q)$.
\end{lemma}

\begin{proof} Set $U = N_{C_G(Z(Q))}(X)$. As $U/X \cong \U_4(2)$, we have $$N_U(XS^*)/X \cong 3^3:\Sym(4).$$ By Lemma \ref{O833}  $N_U(XS^*) = U\cap H$.
Since $S$ acts irreducibly on $X/X'$, we have that $N_{X}(Z_2(S))=X'$. Hence using Lemma \ref{O834},   we see that $Z_2(S)$ has $64 \times 40$ conjugates under the conjugation action of $U$. On the other hand,  the number of conjugates is at most $(3^8-1)/2 = 40 \times 82$.  It follows, again using Lemma~\ref{O834}, that  $ Z_2(S)^U = Z_2(S)^{C_G(Z(Q))}$ and  $U = C_G(Z(Q))$.
\end{proof}

\begin{lemma}\label{O836} Suppose that $H^* \cong \mathrm P\Omega^+_8(3)$ and
$Y$ is a $N_G(Q)$-conjugate of $X$. Let $i \in Y$ be an involution with $iQ \not\in Z(Y/Q)$ and $P$ be the
preimage of $C_{N_G(Q)/X'}(\langle i\rangle, X')$. Then $P \leq N_G(Y)$.
\end{lemma}

\begin{proof}  As $X^\prime$ inverts $Q/Z(Q)$, we have that $X^\prime$ is normal in $N_G(Q)$.  To reach the conclusion of the lemma we may as well
suppose that $Y=X$. We set $\ov {C_G(Z(Q))}= C_G(Z(Q))/Q$  and identify it with a subgroup of $\Sp_8(3)$. Let
$\ov i$ be an involution  in $\ov X$.  Since $Z(\ov X) = \langle \ov j \rangle$ acts fixed-point-freely on
$Q/Z(Q)$ and  $\ov i$ and $\ov{ij}$ are $\ov X$-conjugate,  we have $|[Q/Z(Q),\ov i]| = 3^4$. This shows that in
$\Sp_8(3)$ the group $C_{\ov{C_G(Z(Q))}}(\ov i)$ is contained in the subgroup $\Sp_4(3) \times \Sp_4(3)$ which
preserves the decomposition of the natural $\Sp_8(3)$-module in to a perpendicular sum of two non-degenerate
$4$-spaces. The extraspecial group $\ov X$ contains $55$ involutions and under the action of $\ov{N_H(Q)}$ we see
that $\langle \ov i, \ov j \rangle$ has either 27 or 9 conjugates, depending on whether $3$ divides $|H/H^*|$ or
not. Hence  $|C_{\ov{N_H(Q)}}(\ov i)| = 2^b \cdot 3$ where $0\le b\le 4$.

Choose $E \in \{E_1,E_2,E_3\}$ to be an elementary abelian subgroup of $S^*$ of order $3^6$ as in Lemma
\ref{O832}. Then $E$ normalizes $X$ and, as $\ov{EX} \cong \SL_2(3)*2^{1+4}_+$ corresponds to an end node of the
Dynkin diagram, we have $[\ov X,\ov E]\cong \Q_8$. Furthermore, $C_{\ov X}(\ov{E_iE_j})\cong \Q_8$. Thus by
counting we see that every non-central involution of $\ov X$ is centralized by some $E\in \{E_1,E_2,E_3\}$. In
particular, we may assume that $E$ is chosen so that $[i,E] \leq Q$. Then by Lemma \ref{O832} again we get that
$\ov {QE}$ is a Sylow 3-subgroup of $C_{\ov {N_G(Q)}}( \ov i)$ and  $\ov{N_G(QE)} \leq \ov {N_G(Q) \cap N_G(E)}$
as $E= C_{QE}(E\cap Q)$ is the unique elementary abelian subgroup of order $3^6$ in $QE$. Hence by Lemma
\ref{O832} and Lemma \ref{O833} we have  $\ov{N_G(QE)} \leq \ov{N_H(E)}$.

Let $k \in \ov X$ with $\ov i^k =   \ov j$ and write $$W =
 N_{\Sp_8(3)}(\langle \ov i, \ov j \rangle) = (L_1 \times L_2)\langle k \rangle.$$ where $L_1 \cong \Sp_4(3)$, $L_1^k = L_2$ and $\ov i \in L_1$.

Since $\ov X \le W$ and $\ov X $ does not centralize $\ov i$,  we now see that
$$N_{W^\prime}(\ov X) \approx  (2 ^{1+4}_-\times  2^{1+4}_-).\Alt(5)$$  and  $\ov {N_H(Q)}\cap N_{W^\prime}(\ov X)\ov X \ge \ov X \ov E$.

Suppose that $5$ divides the order of $U = C_{\ov{C_G(Q)}}(\ov i)$. And assume that $U \not \le \ov{N_H(Q)}$.
 Then  the structure of $\Sp_4(3) \times \Sp_4(3)$ and the fact that 9 does not divide the order of $U$ now gives that  $U \leq N_W(\ov X)$ in which case the lemma holds.
\\
So assume that $|U| = 2^a\cdot  3$ for some suitable $a$. Recall from Lemma~\ref{O837+}, we have
$N_{\ov{C_G(Z(Q))}}(\ov X)/\ov X \not\cong \U_4(2)$.

From now on we assume the lemma is false in seek a contradiction. Then  $U \not\leq N_W (\ov X)$.  As $U \le L_1
L_2$, we may project $ C_{\ov{X}}(\ov i)\ov E \approx 2^{1+4}_-.3$ on to the first factor (say) and deduce from
the subgroup structure of $\Sp_4(3)$ and the fact that we know $|U|=2^a\cdot 3$ that $U$ normalizes
$C_{\ov{X}}(\ov i)L_2$. We may therefore assume that  $$1 \not= [U, C_{\ov X}(\ov i)] \cap L_1 \leq
O_2(N_{W^\prime}(\ov X)) \cap L_1.$$  Moreover, as $3$ divides $|U|$,  $([U, C_{\ov X}(\ov i)]\cap L_1)/\langle
\ov i \rangle$ is elementary abelian of order 4. This shows that $N_{C_G(Z(Q))}(X)/X$ contains an elementary
abelian subgroup of order 4, which by Lemma \ref{O835} implies $N_{C_G(Z(Q))}(X) \leq H$. But there is no
$\Alt(4)$ in $N_H(X)/X$ with $EQ/Q$ as a Sylow 3-subgroup and so we have a contradiction. (Recall $N_{H^*}(X)/X
\approx 3^3:2$.) So  $U \leq N_W(\ov X)$.
\end{proof}

\begin{lemma}\label{O837}  Suppose that  $H^* \cong \mathrm P\Omega^+_8(3)$ and $Y$ is $N_G(Q)$-conjugate to $X$ in
$N_G(Q)$. Then $Y$ weakly closed in $N_G(Y)$ with respect to $C_G(Z(Q))$. In particular $|N_G(Q) : N_G(X)|$ is
odd.
\end{lemma}

\begin{proof} Again it suffices to prove the result for $X$.
Suppose that $X^g \le N_G(X)$ with $[X,X^g] \leq X \cap X^g$ and $X \not= X^g$ with $g \in N_G(Q)\setminus H$. By
Lemma~\ref{O837+} we have $N_G(X)/X \not \cong \U_4(2)$.

By Lemma \ref{extra} there are no transvections in $C_G(Z(Q))$ on $X/X^\prime$. Hence, by Lemma~\ref{VO6}(i),
$[X, X^g]$ contains an involution $i$ with $iQ \not\in Z(X/Q)$. Therefore    Lemma \ref{O836} yields that
$C_{N_G(Q)}(\langle i, Q\rangle) \leq N_G(X) \cap N_G(X^g)$.

Assume that $N_{C_G(Z(Q))}(X)/X \cong 3^{1+2}_+:\SL_2(3)$. Then $|X^g  X/X| = 2$, contradicting the fact that there are no transvections on $X/X^\prime$.

Therefore  Lemma \ref{O835} implies  that $N_G(X) \leq H$ and so  $N_{C_G(Z(Q))}(X)/X$ is a subgroup
of $3^3 :\Sym(4)$ and $X^gX/X$ is a fours group. Let $E$ be as in Lemma~\ref{O832} be such that $N_{N_G(Q)}(E)
\le N_H(Q)$. Then, as in the previous lemma, we may assume that $E Q/Q$ centralizes $iQ$. Thus $EQ \le N_G(X^g)
\cap N_G(X)$. From the structure of $N_{H^g}(X^g)/X^g$, we see that  $N_{N_G(X^g)}(EQ)/EQ$ contains an elementary
abelian group of order $9$ and this group is in turn contained in $H$ by Lemma~\ref{O832}.  Thus $(N_G(X) \cap
N_G(X^g))/Q$ contains an elementary abelian group of order $9$ and this group normalizes $X^gX$. Since $N_G(X)/X$
is a subgroup of a group of shape $3^3:\Sym(4)$ an  easy calculation shows that it is impossible for a fours
group to be normalized by an elementary abelian groups of order $9$.
 This contradiction proves the lemma.
\end{proof}

\begin{lemma}\label{O838} Suppose that  $H^* \cong \mathrm P\Omega^+_8(3)$. Then
$N_G(X) = N_G(Q)$.
\end{lemma}

\begin{proof} By Lemmas~\ref{O835} and \ref{O837} we may assume that
$$N_{C_G(Z(Q))}(X)/X \cong 3^{1+2}_+:\SL_2(3)$$ or $N_G(X) = N_H(X)$.  Set $\ov{N_G(Q)} = N_G(Q)/X^\prime$.
\\

We first will show that for involutions $i \in X\setminus X'$  we have

\begin{claim}\label{cl3} $i^{C_G(Z(Q))} \cap N_G(X) \subseteq X$. \end{claim}

\medskip
Assume $i^g \in N_G(X) \setminus X$ for some $g \in N_G(Q)$. By Lemmas \ref{O836}  and \ref{O837}
$(X^g \cap X)/Q$ is isomorphic to a subgroup of $\Q_8$. In particular  $|\ov{X} \cap \ov{X}^g| \leq
4$. As $|C_{\ov{X}}(i^g)| = 16$ by Lemma~\ref{VO6}(i), we see that $N_G(X)/X$ contains a fours group $V=(\ov{X}^g
\cap \overline{N_G(X)})\ov{X}/\ov{X}$ and so  $N_G(X) =N_H(X)$ and we recall that $N_H(X)/X $ is
isomorphic to a subgroup of a groups of shape $3^3.\Sym(4)$. Thus $(N_G(X)/X)/O_3(N_G(X)/X)$ is a subgroup of
$\Sym(4)$. In particular, as there is no elementary abelian group of order 8 in $N_G(X)/X$ and, by Lemma
\ref{O837},  $N_G(X)$ contains a Sylow 2-subgroup of $N_G(Q)$, we have that $|\ov{X} \cap \ov{X}^g| = 4$ and the
preimage of this group in $X/Q$ is quaternion of order 8.  Therefore,   $$[C_{\ov X}(i^g),\ov X^g \cap
\ov{N_G(X)}] \le \ov{X} \cap \ov X^g$$ and  consists only of elements whose preimages have order 4.

 Now the action of $N_G(X)$ on $X/Q$ is uniquely determined. Up to conjugacy, there are exactly two fours
 groups $F_1$ and $F_2$ in $\Sym(4)$ where we assume that $F_1$ is normal. We have
 $|C_{\ov{X}}(F_1)|= |C_{\ov{X}}(F_2)|=8$ from Lemma~\ref{VO6}(ii). Furthermore, also from  Lemma~\ref{VO6},
 the preimage of $C_{\ov{X}}(F_1)$ in $X/Q$ is abelian and so as $X \cap X^g$ is quaternion, we cannot have $V= F_1$.
 Hence  $V=F_2$.  On the other hand, by Lemma~\ref{VO6} again
 $[C_{\ov{X}}(i^g), F_2]$ contains elements whose preimages are involutions. Thus  we also have $V \neq F_2$.
 This contradiction shows that all conjugates of $i$ in $N_G(X)$ are contained in $X$ as claimed.\qedc

Now consider $X^g \cap N_G(X)$ for $g \in N_G(Q)$. By  Lemmas \ref{O836}  and \ref{O837}  there are no
involutions in $(X\cap X^g)\setminus X'$ and there are no involutions in $(X^g \cap N_G(X))\setminus X$ by
\ref{cl3}. Thus $(X^g\cap N_G(X))/Q$ is a subgroup of a quaternion group. It follows that $N_G(X^g)\cap X$ has
index at least $16$ in $X$ for every $g \in N_G(Q)\setminus N_G(X)$. Therefore the number of conjugates of $X$ in
$N_G(Q)$ is $1+ k16$ for some integer $k$. On the other hand, $|N_G(Q):N_G(X)||N_G(X):N_G(Z_2(S^*))| \le
(3^{8}-1)/2$. Hence, as $|N_G(X):N_G(Z_2(S^*))|=64$,  $|N_G(Q):N_G(X)|< 52$.  Thus the number of
conjugates of $X$ in $N_G(Q)$ is  $1$, $17$, $33$ or $49$. The only one of these numbers which divides
$|\Sp_8(3)|$ is $1$. Hence $X$ is normal in $N_G(Q)$.
\end{proof}

\begin{proposition}\label{O83} Suppose that  $H^* \cong \mathrm P\Omega^+_8(3)$. Then $F^\ast(G) \cong \F_2$ or $\M(23)$.
\end{proposition}
\begin{proof} By Lemma \ref{O838}, we have that $X$ is normal in $N_G(Q)$. As $N_G(Q) \not\leq H$ and $N_H(Q)$ contains an element which inverts $Z(Q)$,  Lemma \ref{O835} indicates that
$N_G(Q)/Q$ is an extension of $X$ by $3^{1+2}_+.\GU_3(2)$ or $\U_4(2):2$. Now an application of Lemma~\ref{M231} and Lemma~\ref{F2}  yield
the assertion.
\end{proof}

We now prove Theorem~\ref{main}.

\begin{proof}[Proof of Theorem~\ref{main}] We have already proved the theorem when $p=2$ in Section 3. So we may
now suppose that $p=3$. Lemma~\ref{possible} (ii) indicates that $H^* \cong \G_2(3^a)$ with $a\ge 1$,
$\PSp_4(3)$, $\PSL_4(3)$, $\U_4(3)$, $\Omega_7(3)$ or $\mathrm P\Omega_8^+(3)$. The first two possibilities are
eliminated by Lemmas~\ref{sp43} and \ref{notG2} and the remaining cases are shown to result in the groups listed
in Theorem~\ref{main}(iii) in Propositions~\ref{L43}, \ref{O73}, \ref{O83}.
\end{proof}

\end{document}